\documentclass{article}
\usepackage{import}
\usepackage{amsmath,amsfonts,amsthm, amssymb, mathrsfs, wrapfig, framed, setspace, standalone}
\usepackage[T1]{fontenc}
\usepackage{authblk}

\usepackage[margin=1in]{geometry}
\usepackage{multicol}
\usepackage[skins]{tcolorbox}
\usepackage{dsfont}
\usepackage{fancyhdr}
\usepackage[utf8]{inputenc}
\usepackage[T1]{fontenc}
\usepackage{tikz}
\usepackage{tikz-cd}
\usepackage{url}
\usepackage{setspace}
\usepackage{wasysym}
\usepackage{chngcntr}
\usepackage{mathtools}
\usepackage{graphicx}
\usepackage{lipsum}
\usepackage{enumerate}
\usepackage{parskip}
\usepackage[utf8]{inputenc}
\usepackage[english]{babel}
\usepackage{hyperref}
\usepackage{tikz}
\usetikzlibrary{cd}
\usetikzlibrary{arrows}
\tikzset{
dot/.style = {circle, fill, minimum size=#1,
              inner sep=0pt, outer sep=0pt},
dot/.default = 6pt 
}
\hypersetup{
	pdftitle={Sampling Finite Unit Norm Tight Frames Using Symplectic Geometry},
	pdfauthor={Mason Faldet and Clayton Shonkwiler},
	pdfsubject={frame theory},
	pdfkeywords={frames, random frames, symplectic geometry},
	colorlinks,
    citecolor=blue,
    filecolor=black,
    linkcolor= black,
    urlcolor=blue
}
\usetikzlibrary{positioning}
\usetikzlibrary{petri}
\usetikzlibrary{arrows}

\usepackage[nameinlink,capitalize]{cleveref}

\newcommand{\C}{\mathbb{C}}

\renewcommand{\Im}{\text{Im }}
\newcommand{\R}{\mathbb{R}}

\renewcommand{\H}{\mathcal{H}}

\newcommand{\Z}{\mathbb{Z}}
\newcommand{\N}{\mathbb{N}}

\newcommand{\bpm}{\begin{pmatrix}}
\newcommand{\epm}{\end{pmatrix}}

\renewcommand{\phi}{\varphi}

\renewcommand{\epsilon}{\varepsilon}

\newcommand{\ad}{\text{Ad}}
\newcommand{\F}{\mathcal{F}}

\newcommand{\g}{\mathfrak{g}}

\renewcommand{\>}{\rangle}
\newtheorem{theorem}{Theorem}[section]
\newtheorem{lemma}[theorem]{Lemma}
\newtheorem{corollary}[theorem]{Corollary}
\newtheorem{proposition}[theorem]{Proposition}

\theoremstyle{definition}
\newtheorem{remark}[theorem]{Remark}

\newtheorem{example}[theorem]{Example}

\newtheorem{definition}[theorem]{Definition}

\makeatletter
\@addtoreset{section}{part}
\makeatother
\newlength\mylen

\allowdisplaybreaks

\AddToHook{env/example/begin}{\crefalias{theorem}{example}}
\AddToHook{env/lemma/begin}{\crefalias{theorem}{lemma}}
\AddToHook{env/remark/begin}{\crefalias{theorem}{remark}}

\usepackage[mathlines]{lineno}
\usepackage{mathdots}
\usepackage{stmaryrd}
\usepackage{adjustbox}
\usepackage{array}
\usepackage{algorithm}
\usepackage{algpseudocode}
\usepackage{rotating}
\usepackage{xspace}
\renewcommand{\i}{\sqrt{-1}}
\renewcommand{\u}{\mathfrak{u}}
\newcommand{\framespace}{\F^{d,N}}
\newcommand{\tight}{\framespace_{\boldsymbol{N/d}}}
\newcommand{\fun}{\tight(\boldsymbol{1})}

\renewcommand{\Tilde}{\widetilde}
\newcommand{\eigenlift}{{\tt Eigenlift}\xspace}
\newcommand{\tr}{\operatorname{tr}}
\newcommand{\Ad}{\operatorname{Ad}}
\renewcommand{\ad}{\operatorname{ad}}
\newcommand{\diag}{\operatorname{diag}}
\newcommand{\vol}{\operatorname{vol}}
\makeatletter
\def\ps@pprintTitle{%
  \let\@oddhead\@empty
  \let\@evenhead\@empty
  \let\@oddfoot\@empty
  \let\@evenfoot\@empty
}
\makeatother

\title{Sampling Finite Unit Norm Tight Frames Using Symplectic Geometry}
\author{Mason Faldet}
\author{Clayton Shonkwiler}
\affil{Department of Mathematics, Colorado State University, Fort Collins, CO, USA}

\date{}

\begin{document}

\maketitle

\begin{abstract}
Unit-norm tight frames in finite-dimensional Hilbert spaces (FUNTFs) are fundamental in signal processing, offering optimal robustness to noise and measurement loss. In this paper we introduce the \eigenlift algorithm for sampling random FUNTFs. Our approach exploits the symplectic geometry of the FUNTF space, which we characterize as a symplectic reduction of frame space by a symmetry group. We then define a Hamiltonian torus action on this reduced space whose momentum map induces a fiber bundle structure. The algorithm proceeds by sampling a point from the base space, which is a convex polytope, lifting it deterministically to a point on the corresponding fiber, then acting on this point by a random element of the torus to obtain a random FUNTF. We implement the method in Python and validate it in low-dimensional settings where it is computationally feasible to sample the base polytope via rejection sampling.
\end{abstract}

\section{Introduction}

A \emph{frame} in a Hilbert space $\H$ is a sequence of vectors $\{f_i\}_{i \in I} \subset \H$ such that for every $v \in \H$,
\begin{linenomath*}
    \begin{align*}
    a \|v\|^2 \leq \sum_{i \in I} |\<v,f_i\>|^2 \leq b\|v\|^2
\end{align*}
\end{linenomath*}
for some $0 < a\leq b$ called the \emph{frame bounds}. When $a=b$ the frame is called \emph{tight}. A finite collection $\{f_i\}_{i=1}^N \subset \C^d$ is a frame if and only if $\operatorname{span} \{f_1, \dots, f_N\} = \C^d$; that is, finite frames are equivalent to spanning sets. 

Our main focus is on finite tight frames in which each all of the frame vectors are unit vectors; that is, $\|f_i\| = 1$ for all $i = 1, \dots , N$. These are exactly the \emph{finite unit-norm tight frames} (FUNTFs). FUNTFs have proved useful in a variety of contexts because they provide robust representations of signals with optimal resilience to noise and random erasures of measurements \cite{app1, app2, app3}.

In this paper, we introduce the \eigenlift algorithm for sampling FUNTFs uniformly from a natural probability measure on the space of FUNTFs. This algorithm leverages the toric symplectic structure on the space of FUNTFs described in \cite{shonk}. This structure involves a specific toric action on the frame space and a many-to-one map from the frame space to a convex polytope (the \emph{eigenstep polytope}) that is invariant under the torus action. Heuristically, the process consists of sampling a point from Lebesgue measure on the eigenstep polytope, identifying a frame in the fiber over that point, and applying a random element of the torus to the resulting frame.

We have implemented our algorithm in Python and used this implementation to conduct small-scale experiments. A key challenge lies in efficiently sampling from the eigenstep polytope; in principle this can be done using a standard Markov chain, but in practice convergence seems to be quite slow in high dimensions. For our proof of concept, we sampled the eigenstep polytope by rejection sampling from a bounding box; unfortunately, the rejection probability goes rapidly to 1 as the dimension of the space of FUNTFs increases, so this approach is only practical in low dimensions.

\paragraph{Outline of the paper} \Cref{sec:background} gives a summary of key definitions and results from frame theory and symplectic geometry that are used in this work. We then describe the symplectic structure on a suitable quotient of the space of finite unit-norm tight frames in \cref{sec:symplectic structures on FUNTF space}. In fact, this space is highly symmetric: an open, dense subset admits a Hamiltonian action of a half-dimensional torus, making this subset into a toric symplectic manifold. We describe this toric structure in \cref{sec:toric geometry} and then use it to define the \eigenlift algorithm in \cref{alg:sampling FUNTF}. \cref{sec:implementation} summarizes results from our low-dimensional experiments and the paper concludes with \cref{sec:conclusion}, which includes some discussion and future directions.

\section{Background on Frame Theory and Symplectic Geometry}
\label{sec:background}

\subsection{Concepts From Frame Theory}

\label{frame prelim}
%

 Let $\framespace$ denote the set consisting of all frames of $N$ vectors in $\C^d$. View $\mathcal{F}^{d,N}$ as a subset of $\C^{d\times N}$ by identifying a frame $\{f_i\}_{i = 1}^N \in \framespace$ with the $d \times N$ matrix $F := \begin{bmatrix}
    f_1|& \cdots &|f_N 
\end{bmatrix}$. The frame space $\framespace$ is the complement of the vanishing locus of the real polynomial function
\begin{linenomath*}
    \begin{align*}
    \begin{bmatrix}
        f_1| & \cdots & |f_N
    \end{bmatrix} \mapsto \sum_{\substack{\{i_1, \dots, i_d\} \\ \subset \{1, \dots, N\}} } \left|\det \begin{bmatrix}
        f_{i_1}| & \cdots & |f_{i_d}
    \end{bmatrix}\right|^2,
\end{align*}
\end{linenomath*}
so $\mathcal{F}^{d,N}$ is an open, dense set in $\C^{d,N}$. \par

For any $F \in \framespace$, define three related operators:
\begin{enumerate}[(i)]
    \item The \emph{analysis} operator $\C^d \to \C^N$ defined by $v \mapsto F^*v$, or equivalently
    \begin{align*}
        v \mapsto (\<v,f_1\>, \dots, \<v,f_N\>).
    \end{align*}
    \item The \emph{synthesis operator} $\C^N \to \C^d$ defined by $w \mapsto Fw$, or equivalently 
    \begin{linenomath*}
    \begin{align*}
        (w_1, \dots, w_N) \mapsto \sum_{i = 1}^N w_if_i.
    \end{align*}
    \end{linenomath*}

    \item The \emph{frame operator} $\C^d \to \C^d$ is the composition of the analysis and synthesis operators; that is, it is the operator $v \mapsto FF^* v$, or equivalently
	\[
		v \mapsto \sum_{i=1}^N \<v,f_i\>f_i.
	\] 
\end{enumerate}

We can re-interpret tightness in terms of the frame operator: a frame $F \in \framespace$ is tight if and only if $FF^* = a \mathbb{I}_d$ for some $a > 0$; if we want to emphasize the constant $a$, we will call the frame \emph{$a$-tight}. If for every $i \in \{1, \dots, N\}$,  $\|f_i\|^2 = r$ for some $r > 0$ we say $F$ is \emph{equal norm} and if $r = 1$ we call $F$ \emph{unit norm}. Lastly, we say $F$ is \emph{full spark} if any size $d$ subset $\{f_{i_1}, \dots, f_{i_d}\}$ of $F$ is a basis of $\C^d$.

In this work, we restrict our focus to finite unit norm tight frames (FUNTFs). These frames are highly structured as we will see now.  Let $F = \{f_i\}_{i = 1}^N \in \mathcal{F}^{d,N}$ be a FUNTF. By the cyclic invariance of trace we have
 \begin{linenomath*}
\begin{align}
    N = \sum_{i = 1}^N \|f_i\|^2 =  \tr F^*F = \tr F F^* = \sum_{i = 1}^d \lambda_i,
    \label{trace}
\end{align}
 \end{linenomath*}
where $\{\lambda_i\}_{i=1}^d$ is the spectrum of the frame operator. But since $F$ is tight, $FF^* = a \mathbb{I}_d$ for some $a > 0$, so $\lambda_i = a$ for each $i = 1, \dots,d$. After substituting these equalities into \eqref{trace} we see that $\lambda_i = N/d$. Hence, if $F$ is a FUNTF, then $F$ is $N/d$-tight and the eigenvalues of $FF^*$ are $(\frac{N}{d}, \dots, \frac{N}{d})$. \par 

\begin{figure}[h]
\begin{center}
    \begin{tikzcd}[column sep=tiny, row sep=tiny]
    {\textcolor{blue}{\mu_{N,1}}}  & {\textcolor{blue}{\mu_{N,2}}} & \cdots   & {\textcolor{blue}{\mu_{N,d-2}}}  & {\textcolor{blue}{\mu_{N,d-1}}}  & {\textcolor{blue}{\mu_{N,d}}}\\
    {\textcolor{blue}{\mu_{N-1,1}}}  & {\textcolor{blue}{\mu_{N-1,2}}}  & \cdots  & {\textcolor{blue}{\mu_{N-1,d-2}}}  & {\textcolor{blue}{\mu_{N-1,d-1}}}  & {\mu_{N-1,d}} \\
    {\textcolor{blue}{\mu_{N-2,1}}}  & {\textcolor{blue}{\mu_{N-2,2}}}  & \cdots  & {\textcolor{blue}{\mu_{N-2,d-2}}}  & {\mu_{N-2,d-1}}  & {\mu_{N-2,d}}\\
    \vdots & \vdots & & \vdots & \vdots & \vdots \\
    {\textcolor{blue}{\mu_{N-d+2,1}}}  & {\textcolor{blue}{\mu_{N-d+2,2}}}  & \cdots  & {\mu_{N-d+2,d-2}}  & {\mu_{N-d+2,d-1}}  & {\mu_{N-d+2,d}}\\    
    {\textcolor{blue}{\mu_{N-d+1,1}}}  & {\mu_{N-d+1,2}}  & \cdots  & {\mu_{N-d+1,d-2}}  & {\mu_{N-d+1,d-1}}  & {\mu_{N-d+1,d}}\\{\mu_{N-d,1}}  & {\mu_{N-d,2}}  & \cdots  & {\mu_{N-d,d-2}}  & {\mu_{N-d,d-1}}  & {\mu_{N-d,d}}\\
    \vdots & \vdots & & \vdots & \vdots & \vdots \\ {\mu_{d,1}}  & {\mu_{d,2}}  & \cdots  & {\mu_{d,d-2}}  & {\mu_{d,d-1}}  & {\mu_{d,d}} \\
    {\mu_{d-1,1}}  & {\mu_{d-1,2}}  & \cdots  & {\mu_{d-1,d-2}}  & {\mu_{d-1,d-1}}  & 0 \\
    {\mu_{d-2,1}}  & {\mu_{d-2,2}}  & \cdots  & {\mu_{d-2,d-2}}  & 0 & &  \\
    \vdots &  \vdots & \iddots & \iddots \\
    \mu_{2,1}  & \mu_{2,2}  & 0 \\
    \mu_{1,1}  & 0 \\
    0
    \end{tikzcd}
\end{center}
\caption{Eigensteps of a generic FUNTF which, by \eqref{nonincreasingSteps} and \eqref{Weyls}, are such that the columns, rows, and upward diagonals are nonincreasing. Moreover, the values in blue are $N/d$ and the sum of the entries in the row corresponding to $S_k$ is equal to $k$. The independent eigensteps are all those colored black, excluding the rightmost nonzero term in each row.}
\label{fig:eigensteps}
\end{figure}

 The eigenvalues of the \emph{partial frame operators} $S_k := f_1f_1^* + \cdots f_kf_k^*$ will also be of importance. By virtue of being a sum of rank $1$ matrices, $\operatorname{rank} S_k \leq \min\{k,d\}$ and $S_k$ has at most $\min \{k,d\}$ nonzero eigenvalues. For a fixed $k = 1, \dots, N$, let 
 \begin{linenomath*}
     \begin{align}
     \mu_{k,1} \geq \mu_{k,2} \geq \cdots \geq \mu_{k_d}
     \label{nonincreasingSteps}
    \end{align}
 \end{linenomath*}
 be the eigenvalues of $S_k$ arranged in nonincreasing order. The values $\mu_{kj}$ are called \emph{eigensteps}~\cite{fickus}. Because $S_k$ is a rank 1 perturbation of $S_{k-1}$ i.e. $S_k = S_{k-1} + f_kf_k^*$, Weyl's Inequality applies so that the eigensteps satisfy the interlacing inequalities 
 \begin{linenomath*}
     \begin{align}
         \cdots \geq \mu_{k,j} \geq \mu_{k-1,j} \geq \mu_{k,j+1} \geq \mu_{k-1,j+1} \geq \cdots
         \label{Weyls}
     \end{align}
 \end{linenomath*}
 for each $k = 2, \dots, N$ and $j = 1, \dots, \min\{k,d\}$. In \cref{fig:eigensteps} the eigensteps are arranged into a table in which the rows, columns, and upward diagonals are nonincreasing as a consequence of \eqref{nonincreasingSteps} and \eqref{Weyls}. Since $S_N = FF^*$, we see that $\mu_{N,i} = N/d$ for each $i = 1,\dots, d$. This, along with \eqref{Weyls}, requires all the values colored blue in Figure~\ref{fig:eigensteps} to also be $N/d$. Lastly, applying the cyclic invariance of trace as in \eqref{trace} we have 
 \begin{linenomath*}
     \begin{align}
     \sum_{j = 1}^{\min\{k,d\}} \mu_{k,j} = \sum_{j = 1}^k \|f_j\|^2 = k
     \label{rowsum}
 \end{align}
 \end{linenomath*}
 so that the values in the row corresponding to $S_k$ sum to $k$. It follows that the independent eigensteps are those colored black in Figure~\ref{fig:eigensteps}, except for the rightmost eigenstep in each row, which can be expressed in terms of the other entries in the row using~\eqref{rowsum}. Thus, the eigensteps of a FUNTF consisting of $N$ vectors in $\C^d$ can be parameterized by $(N-d-1)(d-1)$ free variables which we call the \emph{independent eigensteps} (cf.~\cite{hagaPolytopesEigenstepsFinite2016}). Let $\overline{\mathcal{P}}_{d,N}$ be the \emph{eigenstep polytope}; that is, the polytope of independent eigensteps.
 
\begin{example}
\label{35example}
Let $d = 3$ and $N = 5$. The eigenstep table of a FUNTF in $\F^{3,5}$ is of the following form:
\begin{center}
        \begin{tikzcd}[column sep=tiny, row sep=tiny]
        5/3 & 5/3 &  5/3 \\
        5/3 &  5/3 &  2/3 \\
        5/3 & x_1 & 3 - 5/3 - x_1\\
        x_2 & 2 - x_2 & 0 \\
        1 & 0 &  0
    \end{tikzcd}
\end{center}
Notice that there are $(5-3-1)(2-1) = 2$ free variables. All entries are nonnegative and we have the following inequalities: $5/3 \geq x_1 \geq 3 - 5/3 - x_1 \geq 0$ and  $x_2 \geq 2 -x_2\geq 0$ from nonincreasing rows, $x_2 \geq x_1 \geq 2/3$ and $1 \geq 2 - x_2 \geq 3 - 5/3 -x_1$ from nonincreasing upward diagonals, and $5/3 \geq x_2 \geq 1$, $5/3 \geq x_1 \geq 2 -x_2$, and $2/3 \geq 3 - 5/3 -x_1$ from the nonincreasing columns. The corresponding eigensteop polytope is shown in \cref{fig:53polytope}.
\begin{figure}[h]
	\centering
		\includegraphics[height=2in]{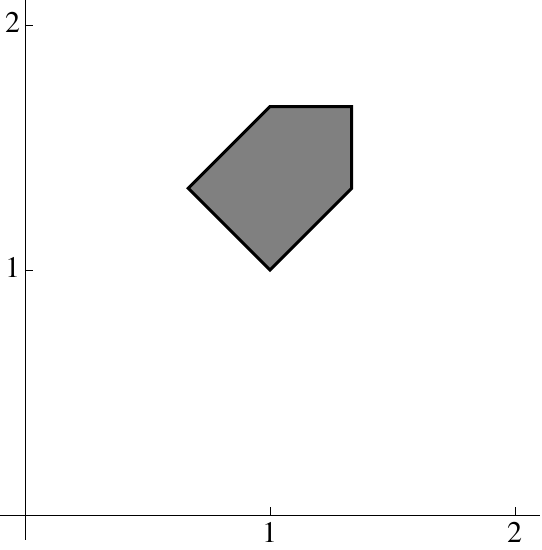}
	\caption{The eigenstep polytope $\overline{\mathcal{P}}_{3,5}$; each point $(x_1,x_2)$ corresponds to a valid way of filling in the eigenstep table for FUNTFs in $\F^{3,5}$.}
	\label{fig:53polytope}
\end{figure}
\end{example}
We have established that the collection of eigensteps corresponding to a FUNTF are highly structured and involve many relations. Nonetheless, there are FUNTFs in $\F^{d,N}$ for any $N \geq d$~\cite{app2,zimmermannNormalizedTightFrames2001a} (see also~\cite{benedettoFiniteNormalizedTight2003,casa}).

To conclude our frame theory summary, we state two results upon which this work is based. The first result comes from a paper by Needham and Shonkwiler \cite{shonk}, in which they prove that a uniform sample from the space of FUNTFs will be full-spark with probability 1.
\begin{theorem}
\label{full-spark}
    For any $N \geq d \geq 1$, the full-spark frames have full measure inside the space of FUNTFs in $\C^d$. 
\end{theorem}
\noindent We use the symplectic framework they establish in their proof of this result to derive our algorithm. The second result is an algorithm from the work of Cahill et al. \cite{fickus} for constructing a finite frame with prescribed eigensteps, which we state as applied to FUNTFs.
\begin{theorem}[{Cahill et al.~\cite[Theorem~2]{fickus}}]
    For any doubly indexed nonnegative sequence $\{\{\mu_{k,j}\}_{j = 1}^d \}_{k = 1}^N$ which satisfies
    \begin{enumerate}
        \item $\mu_{k,j} = 0$ for $j = k + 1, \dots, d$ whenever $k < d$.
        \item Weyl's Inequalities \eqref{Weyls}
        \item $\mu_{N,j} = N/d$ for each $j = 1, \dots, d$
        \item $\displaystyle \sum_{j =1}^{\min\{k,d\}} \mu_{k,j} = k$ for each $k = 1, \dots, N$, 
    \end{enumerate}
    every FUNTF $F = \{f_n\}_{n = 1}^N$ whose eigensteps are $\{\{\mu_{k,j}\}_{j = 1}^{d}\}_k^N$ can be constructed by the following process: For each $k = 1, \dots, N$, consider the polynomial 
    \begin{linenomath*}
    \begin{align*}
        p_k(x) := \prod_{j=1}^{d}(x-\mu_{k,j}).
    \end{align*}
    \end{linenomath*}
    Take any $f_1 \in \C^d$ such that $\|f_1\| = 1$. For each $k = 1, \dots, N-1$, choose any $f_{k+1}$ such that 
    \begin{linenomath*}
        \begin{align}
        \|P_{k;\mu} f_{k+1}\|^2 = - \lim_{x \to \mu} \frac{p_{k+1}(x)}{p_k(x)}
        \label{limit}
    \end{align}
    \end{linenomath*}
    for all $\mu \in \{\mu_{k,j}\}_{j=1}^d$, where $P_{k;\mu}$ denotes the orthogonal projection onto the eigenspace $\ker(\mu I - S_k)$ of $S_k = f_1f_1^* + \dots + f_kf_k^*$. The limit in (\ref{limit}) exists and is nonpositive. Conversely, any $F$ constructed in this way has eigensteps $\{\{\mu_{k,j}\}_{j = 1}^{d} \}_{k = 1}^N$.
\end{theorem}
\subsection{Concepts From Symplectic Geometry}
 In this work, we equip the space of FUNTFs with a symplectic structure, which is fundamental to our algorithm. Our method centers on inverting the momentum map of a Hamiltonian toric action on the symplectic manifold of FUNTFs. Below is a brief introduction to relevant concepts from symplectic geometry; see \cite{McDuff} for a comprehensive introduction to symplectic geometry.
 
A \emph{symplectic manifold} $(M,\omega)$ is a smooth (real) manifold $M$ endowed with a nondegenerate closed $2$-form $\omega$, called a \emph{symplectic form}. For $p \in M$ and $X,Y \in T_pM$, we use $\omega_p(X,Y) \in \R$ to denote the evaluation of $\omega$ on the tangent vectors $X$ and $Y$. 

Nondegeneracy means that at each point $p \in M$, the bilinear form $\omega_p$ is nondegenerate; that is, if $\omega_p(X,Y) = 0$ for every $Y \in T_pM$, then $X = 0$. One consequence of nondegeneracy is that that $M$ must be even-dimensional. 

The requirement that $\omega$ be closed means that $d\omega = 0$, where $d$ is the exterior derivative on forms.

\begin{example}
     The prototypical example of a symplectic manifold is $n$-dimensional complex space $\C^n \approx \R^{2n}$ with the \emph{standard symplectic form} $\omega$. Using coordinates $(x_1 + y_1\sqrt{-1}, \dots, x_n + y_n\i)$ with $x_i,y_i \in \R$, $\omega$ is defined as
\begin{linenomath*}
    \begin{align*}
        \omega = dx_1 \wedge dy_1 + \cdots + dx_n \wedge dy_n.
    \end{align*}
\end{linenomath*}
For $p \in \C^n$ and  $Z = (z_1, \dots, z_n), W = (w_1, \dots, w_n) \in \C^n \approx T_p\C^n$, $\omega$ can be written in complex coordinates as 
\begin{linenomath*}
    \begin{align*}
        \omega_p(Z,W) = -\operatorname{Im}(\overline{w}_1z_1 + \cdots + \overline{w}_nz_n) = -\operatorname{Im}(W^*Z)
    \end{align*}
\end{linenomath*}
where $\operatorname{Im}$ denotes the imaginary part of a complex number. 

We can extend the standard symplectic form on $\C^n$ to a symplectic form on the manifold $\C^{d\times n}$ of $d \times N$ complex matrices by simply composing the standard symplectic form with the reshaping operator $\C^{d \times N} \to \C^{dN}$. In particular, for $F \in \C^{d \times N}$ and $X,Y \in \C^{d\times N} \approx T_F\C^{d \times N}$
\begin{linenomath*}
    \begin{align*}
        \omega_F(X,Y)  = - \operatorname{Im}(\operatorname{tr}(Y^* X)). 
    \end{align*}
\end{linenomath*}
Recall, $\framespace$ is an open subset of $\C^{d \times N}$ so that $\omega$ restricts to $\framespace$ which makes the space of $d \times N$ frames a symplectic manifold. 
\label{prototype}
\end{example}

As in so many areas of geometry, we are often particularly interested in symplectic manifolds with symmetries; that is, those symplectic manifolds admitting a smooth action of a Lie group which preserves the symplectic form in an appropriate sense. In this setting, the natural actions are \emph{Hamiltonian actions}, which we now describe. To build intuition, we start by considering a smooth action $\psi : \R \to \operatorname{Diff}(M)$ of the Lie group $\R$ on a symplectic manifold $(M, \omega)$. We use the notation $\psi_t$ to indicate the diffeomorphism $\psi(t) \in \operatorname{Diff}(M)$. The action induces a vector field $X_\psi$ on $M$ which is given at a point $p \in M$ by
\begin{linenomath*}
    \begin{align*}
        (X_\psi)_p = \frac{d}{dt}\bigg|_{t = 0} \psi_t(p).
    \end{align*}
\end{linenomath*}
If the vector field $X_\psi$ is related to the differential of a smooth function $H : M \to \R$ by
\begin{linenomath*}
    \begin{align}
    \omega(X_{\psi}, \hspace{1pt}\cdot\hspace{1pt}) = dH,
    \label{hamilton-raise-index}
    \end{align}
\end{linenomath*}
then $X_\psi$ is called a \emph{Hamiltonian vector field} and $\psi$ a \emph{Hamiltonian $\R$-action.} The relation \eqref{hamilton-raise-index} leads to many desirable consequences. For example, the flow of $X_\psi$---which, by definition, is $\psi$ itself---preserves the symplectic form and, the Lie derivative of $H$ with respect to $X_\psi$ is zero. Together, these imply that $H$ is constant on orbits of $\psi$. 

Now we consider the general case of a smooth action $\psi : G \to \operatorname{Diff}(M, \omega)$ where $G$ is any Lie group with Lie algebra $\g$ and $(M,\omega)$ is a symplectic manifold. We say $\psi$ is a \emph{Hamiltonian $G$-action} if there exists a map $\Phi: M \to \g^*$ (analogous to the function $H$ in the $G = \R$ case above) such that for any $\xi \in \g$, the function $\Phi^\xi:M \to \R$ defined $\Phi^\xi(p) := \Phi(p)(\xi)$ and the vector field $X^\xi$ defined by
\begin{linenomath*}
    \begin{align*}
    X^\xi_p = \frac{d}{dt}\bigg|_{t = 0} (\exp t \xi) \cdot p 
\end{align*}
\end{linenomath*}
are such that $\omega(X^\xi, \hspace{1pt} \cdot \hspace{1pt}) = d\Phi^\xi$, where $\exp: \g \to G$ is the exponential map. Moreover, we require $\Phi$ to be equivariant with respect to the $G$-action on $M$ and the coadjoint action of $G$ on $\g^*$, meaning 
\begin{linenomath*}
    \begin{align}
        \Ad^*_g(\Phi(p)) = \Phi (g \cdot p).
        \label{equivariance}
    \end{align}
\end{linenomath*}
We call $\Phi$ the \emph{momentum map} of the Hamiltonian $G$-action. In essence, a $G$-action $\psi$ is Hamiltonian if the restriction of $\psi$ to any 1-parameter subgroup results in a Hamiltonian $\R$-action on $(M,\omega)$. This becomes clear upon noticing that the quantifier ``for all $\xi \in \g$'' above parameterizes all 1-parameter subgroups of $G$. The momentum map provides the associated Hamiltonian function $\Phi^\xi$ and vector field $X^\xi$. The importance of equivariance will become apparent below. 

Momentum maps of Hamiltonian actions are the means by which we obtain coordinates on the space of FUNTFs that are integral to our sampling scheme. Moreover, momentum maps give rise to symplectic quotients, as first shown by Meyer~\cite{meyerSymmetriesIntegralsMechanics1973} and Marsden and Weinstein~\cite{marsdenReductionSymplecticManifolds1974}:
\begin{theorem}[Marsden–Weinstein–Meyer Theorem for Regular Values]
\label{MW}
    Let $(M, \omega)$ be a symplectic manifold with a Hamiltonian $G$-action, let $\Phi: M \to \g^*$ be a momentum map for the action, and let $\chi \in \g^*$ be a regular value of $\Phi$ such that $G$ acts freely on the level set $\Phi^{-1}(\chi)$. Then the manifold
    \begin{linenomath*}
        \begin{align*}
            M\sslash_{\chi} G := \Phi^{-1}(\chi)/G,
        \end{align*}
    \end{linenomath*}
    called the \emph{symplectic reduction} or \emph{symplectic quotient} over $\chi$, admits a symplectic structure $\omega_{\text{red}}$ which is uniquely characterized by the equation
    \begin{linenomath*}
        \begin{align*}
            q^*\omega_{\text{red}} = \iota^*\omega,
        \end{align*}
    \end{linenomath*}
    where $q: \Phi^{-1}(\chi) \to M \sslash_{\chi} G$ is the quotient map and $\iota : \Phi^{-1}(\chi) \to M$ is the inclusion map.
    \label{regularvalue}
\end{theorem}
Notice that, by the equivariance of the momentum map \eqref{equivariance} and the fact that $G$ acts on $\Phi^{-1}(\chi)$, it follows that $\chi$ is necessarily a fixed point in the coadjoint action of $G$ on $\g^*$. Specifically, for every $g \in G$ we have 
\begin{center}
\begin{tikzcd}[column sep=small]
\Phi^{-1}(\chi) \subset M \arrow[dd,maps to,"\Phi"'] \arrow[rr,maps to, "g \cdot (\hspace{2pt})"] &  & g\cdot \Phi^{-1}(\chi) \subset \Phi^{-1}(\chi) \arrow[dd,maps to, "\Phi"] \\
                                                &  &         \\
\{\chi\} \subset \mathfrak{g}^* \arrow[rr,maps to,"\Ad_g^*"']      &  & \Ad^*_g(\{\chi\}) = \{\chi\}                              
\end{tikzcd}
\end{center}

\noindent This diagram also illustrates the converse: if $\chi \in \g^*$ is a fixed point of the coadjoint action of $G$ on $\g^*$ then the level set $\Phi^{-1}(\chi)$ is $G$-invariant. 

More generally, the equivariance of $\Phi$ implies that $G$ acts on the preimage of any coadjoint orbit,
\begin{linenomath*}
    \begin{align*}
        \mathcal{O}_\chi := \{\xi \in \g^*: \xi = \Ad_g^*\chi \text{ for some } g\in G\}.
    \end{align*}
\end{linenomath*}
In the case that $\chi$ is a regular value of $\Phi$ that is not a fixed point of the coadjoint action, \cref{regularvalue} generalizes to a quotient over the coadjoint orbit $\mathcal{O}_\chi$. 
\begin{theorem}[Marsden–Weinstein–Meyer Theorem for Coadjoint Orbits] 
Let $M, \omega, G,$ and $\Phi$ be as above and $\chi \in \g^*$ be a regular value. If the action of $G$ on $\Phi^{-1}(\mathcal{O}_\chi)$ is free, then the manifold
    \begin{linenomath*}
        \begin{align*}
            M\sslash_{\mathcal{O}_\chi} G := \Phi^{-1}(\mathcal{O}_\chi)/G,
        \end{align*}
    \end{linenomath*}
   admits a symplectic structure $\omega_{\text{red}}$ which is uniquely characterized by the equation
    \begin{linenomath*}
        \begin{align*}
            q^*\omega_{\text{red}} = \iota^*\omega,
        \end{align*}
    \end{linenomath*}
    where $q: \Phi^{-1}(\mathcal{O}_\chi) \to M \sslash_{\mathcal{O}_\chi} G$ is the projection map and $\iota : \Phi^{-1}(\mathcal{O}_\chi) \to M$ is the inclusion map.
    \label{regularvalue2}
\end{theorem}

We state one final useful result on symplectic reduction due to Sjamaar and Lerman~\cite{S-L}, which concerns the symplectic quotient over a singular value $\chi \in \g^*$ of the momentum map. In this more general case, the quotient isn't a manifold, but it still has a natural reduced symplectic structure. In particular, they showed that the symplectic quotient over $\chi$ is a symplectic stratified space, which is a disjoint union of symplectic manifolds that fit together nicely. Remarkably, if the momentum map is proper, then the symplectic stratified space has a unique open, connected, and dense piece called a \emph{stratum} which is a manifold.

\begin{theorem}
    \label{s-k}
    Let $(M, \omega)$ be a symplectic manifold with a Hamiltonian $G$-action, let $\Phi: M \to \g^*$ be a momentum map for the action and let $\chi \in \g^*$. The symplectic reduction 
    \begin{linenomath*}
        \begin{align*}
            M \sslash_{\mathcal{O}_\chi}G := \Phi^{-1}(\mathcal{O}_\chi)/G
        \end{align*}
    \end{linenomath*}
    is a symplectic stratified space with symplectic structure $\omega_{\text{red}}$ characterized by $q^*\omega_{\text{red}} = \iota^*\omega$. If $\Phi$ is proper (for example, if $G$ is compact), then the reduction has a unique open stratum which is a manifold and is connected and dense.
\end{theorem}

\section{Symplectic Structure On FUNTFs}\label{sec:symplectic structures on FUNTF space}

In \Cref{prototype} we noted that the space  $\framespace$ of frames  consisting of $N$ vectors in $\C^d$ is a symplectic manifold. Its symplectic structure is obtained by pulling back the standard symplectic form on $\C^{d \times N}$ via the inclusion map $\framespace \hookrightarrow \C^{d,N}$. Let $\fun \subset \framespace$ denote the set of unit norm tight frames in $\framespace$. This notation aligns with that of \cite{shonk}, where the subscript $\boldsymbol{N/d}$ specifies that frames in this space have frame operator spectrum $(N/d, \dots, N/d) \in \R^d$, and the $\boldsymbol{1}$ in parentheses indicates that each frame vector has squared norm equal to 1. To obtain a natural symplectic structure on FUNTFs, we pass to the quotient of $\fun$ by a certain group of symmetries. 

We derive the desired symplectic structure in two steps. First, we construct the space of $N/d$-tight frames within $\framespace$ as a symplectic reduction. Next, we obtain the space of unit norm tight frames as the symplectic reduction of the $N/d$-tight frames from step 1. From this point onward, we denote the space of $N/d$-tight frames by $\tight$. 
\subsection{Symplectic Structure on \texorpdfstring{$N/d$}{N/d}-Tight Frames}
We will show that $\tight$ is the preimage of a regular value in $\u(d)^*$ under the momentum map of the Hamiltonian $U(d)$-action on $\C^{d \times N}$ given by left multiplication. It then follows from \cref{MW} that $\tight/U(d)$ is a symplectic quotient and thus inherits a natural symplectic structure. 

We begin by identifying $\u(d)^*$ with the space of $d\times d$ Hermitian matrices, $\H(d)$, via the isomorphism $\alpha : \H(d) \to \u(d)^*$ which maps a matrix $\xi \in \H(d)$ to the linear functional 
\begin{linenomath*}
    \begin{align}\label{alpha}
        \alpha_\xi : \u(d) & \to \R\\
        \eta & \mapsto \frac{\i}{2} \tr(\eta^*\xi) .\nonumber
    \end{align}
\end{linenomath*}
Notice that $\alpha$ is simply the map $\u(d) \to \H(d)$ given by $\eta \mapsto -\i/2 \eta$ followed by an application of the Riesz representation theorem with respect to the standard Frobenius inner product on $\u(d)$. Next, we define an action of $U(d)$ on $\H(d)$, referred to as the adjoint action, given by conjugation: 
\begin{linenomath*}
    \begin{align*}
        \Ad: U(d) \times \H(d) & \to \H(d)\\
        (A,\xi) & \mapsto A \xi A^*.
    \end{align*}
\end{linenomath*}
The action $\Ad$ and the isomorphism $\alpha$ are compatible in the following sense:
\begin{lemma}
    \label{equivarient identification}
    The isomorphism $\alpha$ is equivariant with respect to adjoint action on $\H(d)$ and the coadjoint action on $\u(d)^*$, i.e. for each $A \in U(d)$ the following commutes
    \begin{equation}
        \begin{tikzcd}
            \H(d) \arrow[r,"\Ad_A"] \arrow[d, "\alpha"']&  \H(d) \arrow[d,"\alpha"]\\
            \u(d)^* \arrow[r,"\Ad_A^*"] &  \u(d)^*
        \end{tikzcd}
        \label{equivarience of alpha}
    \end{equation}
    It follows that, under $\alpha$, the coadjoint orbit of any element in $\u(d)^*$ is identified with the collection of Hermitian matrices with a fixed spectrum.
\end{lemma}
    \begin{proof}
        Let $A \in U(d)$ and $\xi \in \H(d)$. We readily check that the diagram \eqref{equivarience of alpha} commutes:
        \begin{linenomath*}
        \begin{align*}
            \alpha_{\Ad_A(\xi)}(\eta) = \frac{\i}{2}\tr(\eta^*A\xi A^*) = \frac{\i}{2}\tr((A^*\eta A)^*\xi) = \alpha_\xi(A^*\eta A) = \Ad_A^*\alpha_\xi(\eta)
        \end{align*}
        \end{linenomath*}
        for any $\eta \in \u(d)$. 
		
		To see the second statement holds, let $\xi \in \H(d)$. Since $\xi$ is Hermitian, it has a purely real spectrum $\lambda$ and is unitarily diagonalizable. Hence, there exists an $A \in U(d)$ for which
        \begin{linenomath*}
            \begin{align*}
                \xi = A (\diag(\lambda)) A^* = \Ad_A(\diag(\lambda)).
            \end{align*}
        \end{linenomath*}
       Denote the adjoint orbit of $\xi$ in $\H(d)$ by $\mathcal{O}_\xi$. Notice that $\mathcal{O}_\xi$ is equal to the orbit of $\diag(\lambda)$ in $\H(d)$, which we denote by $\mathcal{O}_\lambda$, and consists of all $d\times d$ Hermitian matrices whose spectrum is $\lambda$. Then on the level of sets, equivariance gives 
        \begin{linenomath*}
            \begin{align*}
                \alpha(\mathcal{O}_\lambda) =  \alpha(\mathcal{O}_\xi) =  \{\alpha_{\Ad_A\xi} : A \in U(d)\}
                = \{\Ad^*_A \alpha_{\xi} : A \in U(d)\} = \mathcal{O}_{\alpha_\xi}, 
            \end{align*}
        \end{linenomath*}
        where $\mathcal{O}_{\alpha_\xi}$ denotes the coadjoint orbit of $\alpha_{\xi}$ in $\u(d)^*$. Since $\alpha$ is an isomorphism, all coadjoint orbits in $\u(d)^*$ take the form $\mathcal{O}_{\alpha_\xi}$. 
    \end{proof}
With the identification $\alpha$, we can now express the momentum map for the left multiplication action of $U(d)$ on $\mathbb{C}^{d \times N}$ more cleanly. Notably, $-N/d \cdot \mathbb{I}_d \in \mathcal{H}(d)$ is a regular value of the momentum map, and its preimage is precisely $\tight$.
\begin{theorem}
\label{reduced tight frames}
The action of $U(d)$ on $\C^{d\times N}$ by left multiplication is Hamiltonian with momentum map
    \begin{linenomath*}
        \begin{align*}
            \Phi_{U(d)}: \C^{d \times N} &\to \H(d) \approx \u(d)^*\\
            F &\mapsto -FF^*
        \end{align*}
    \end{linenomath*}
The regular values of $\Phi_{U(d)}$ are the negative definite matrices in $\H(d)$. It follows that the space $\tight/U(d)$ is the symplectic quotient 
    \begin{linenomath*}
        \begin{align*}
            \C^{d \times N} \sslash_{-N/d\cdot \mathbb{I}_d}U(d).
        \end{align*}
    \end{linenomath*}
In particular, it has a natural symplectic structure.
\end{theorem}
A direct computation, as demonstrated in \cite{shonk}, shows that this action is Hamiltonian, with momentum map $\Phi_{U(d)}$ whose regular values are the negative definite matrices in $\mathcal{H}(d)$. Notice that $\Phi_{U(d)}$ maps a frame to the negative of its frame operator, thus the preimage of the regular value $-N/d \cdot \mathbb{I}_d$ is $\tight$. Additionally, $U(d)$ acts on $\tight$: for any $A \in U(d)$ and $F \in \tight$, we have
\begin{linenomath*}
    \begin{align*}
        \Phi_{U(d)}(AF) = -AFF^*A^* = -A (-N/d \cdot \mathbb{I}_{d}) A^* = -N/d \cdot \mathbb{I}_{d}.
    \end{align*}
\end{linenomath*}
Finally, the action on the fiber is free, as $AF = F$ implies $A = \mathbb{I}_d$. 
Thus, by \cref{MW}, the quotient $\tight/U(d)$ is a symplectic manifold consisting of unitary equivalence classes of $N/d$-tight frames.

It will be useful to identify the symplectic manifold $\tight/U(d)$ with a subset in $\H(N)$ that corresponds to a coadjoint orbit in $\u(N)^*$. Specifically, let ${\mathcal{O}}_{N/d} \subset \H(N)$ denote the set of all $N \times N$ Hermitian matrices whose first $d$ eigenvalues (in non-increasing order) are $N/d$, with the remaining $N-d$ eigenvalues all equal to $0$. By \cref{equivarient identification}, the isomorphism $\alpha$ defined in \eqref{alpha} identifies ${\mathcal{O}}_{N/d}$ with a coadjoint orbit $\mathcal{O}_\beta \subset \u(N)^*$ for some $\beta \in \u(N)^*$. Like all coadjoint orbits, $\mathcal{O}_\beta$ carries a natural symplectic structure, which pulls back along $\alpha$ to a symplectic structure on ${\mathcal{O}}_{N/d}$.

Elements of the tangent space to $\mathcal{O}_\beta$ at $\beta \in \u(N)^*$ are of the form $\ad^*_\xi \beta$ for $\xi \in \u(N)$ where $\ad^*$ is the coadjoint representation of $\u(N)$ on $\u(N)^*$ which is obtained as the derivative of the coadjoint action $\Ad^*$ of $U(N)$ on $\u(N)^*$. The natural symplectic form on $\mathcal{O}_\beta$, called the \emph{Kirillov–Kostant–Souriau (KKS) symplectic form}, is given by 
\begin{linenomath*}
    \begin{align*}
        \omega_\beta^{\text{KKS}}(\ad^*_\xi \beta, \ad^*_{\xi'}\beta) = \beta([\xi,\xi']),
    \end{align*}
\end{linenomath*}
where $[\cdot, \cdot]$ is the Lie bracket on $\u(N)$. 

Recall from the proof of \cref{equivarient identification} that ${\mathcal{O}}_{N/d}$ is the orbit of $\diag(N/d, \dots, N/d,0, \dots, 0)$ under the proper and free adjoint action of $U(N)$ on $\H(N)$, hence ${\mathcal{O}}_{N/d}$ is a properly embedded submanifold of $\H(N)$. Consequently, the restriction of the linear isomorphism $\alpha$ to ${\mathcal{O}}_{N/d}$ is a diffeomorphism from ${\mathcal{O}}_{N/d}$ to $\mathcal{O}_\beta$, and thus $\alpha^*\omega^{\text{KKS}}$ is a symplectic form on ${\mathcal{O}}_{N/d}$. 

In \cite{shonk}, Shonkwiler and Needham proved that the space $\tight/U(d)$, equipped with the reduced symplectic structure, is symplectomorphic to the submanifold ${\mathcal{O}}_{N/d} \subset \H(N)$ with the symplectic form $\alpha^*\omega^{\text{KKS}}$ via the mapping
\begin{linenomath*}
    \begin{align}
        [F] & \mapsto F^*F
        \label{Symplectomorphism}
    \end{align}
\end{linenomath*}
which sends a unitary equivalence class to the Gram matrix of any of its representatives. 

\subsection{Second Symplectic Reduction}
The $N$-torus $U(1)^N$ can be identified with the standard maximal torus in $U(N)$, which is the subgroup
\begin{linenomath*}
    \begin{align*}
        \mathbb{T} = \big\{\diag\big(e^{\sqrt{-1}\theta_1}, \dots, e^{\i \theta_N}\big): \theta_i \in [0,2\pi)\big\}.
    \end{align*}
\end{linenomath*}
Under this identification, $U(1)^N \approx \mathbb{T}$ acts on $\C^{d\times N}$ by right multiplication. We will see momentarily that this action preserves the spectra of frame operators, thus it restricts to an action on the submanifold $\tight \subset \C^{d \times N}$. 

Define the $U(1)^N$-action on $\C^{d \times N}$ by $D \cdot F := FD^* = F\overline{D}$ where $F \in \C^{d \times N}$ and $D \in \mathbb{T}$. Right multiplication by the Hermitian adjoint is chosen to ensure that this action aligns with the action of the full unitary group $U(N)$ on $\C^{d\times N}$. In particular, if $A_1, A_2 \in U(N)$ and $F \in \C^{d \times N}$, then 
\begin{linenomath*}
    \begin{align*}
        A_1 \cdot (A_2 \cdot F) = A_1 \cdot FA_2^* = FA_2^*A_1^* = F(A_1A_2)^* = (A_1A_2) \cdot F
    \end{align*}
\end{linenomath*}
showing that $A \cdot F := FA^*$ indeed defines a $U(N)$-action on $\C^{d\times N}$. Without the Hermitian adjoint in the $U(1)^N$-action definition, an extension to a full $U(N)$-action would not be possible, as generic matrices in $U(N)$ do not commute. 

A straightforward computation shows that this action preserves the spectrum of a frame operator. For $D \in \mathbb{T}$ and $F \in \C^{d \times N}$, the frame operator of $D \cdot F = FD^*$ is 
\begin{linenomath*}
\begin{align*}
    (FD^*)(FD^*)^* = FD^*DF^* = FF^*
\end{align*}
\end{linenomath*}
which matches the frame operator of $F$. Thus, the action preserves the frame operator spectrum, and in particular, $U(1)^N$ acts on $\tight$. 

By the associativity of matrix multiplication, the left action of $U(d)$ on $\tight$ commutes with the right action of $\mathbb{T}$ on $\tight$. This implies that the $\mathbb{T}$-action on $\tight$ descends naturally to a well-defined action on the quotient $\tight/U(d)$, given by $D \cdot [F] := [FD^*]$. However, the descended action on $\tight/U(d)$ is not effective. In fact, the isotropy group at any unitary equivalence class $[f]$ contains the 1-parameter subgroup $\{e^{i\theta} \mathbb{I}_N\}_\theta \subset \mathbb{T}$. To see this, let $e^{i \theta} \in U(1)$ and  $ F \in \tight$. Then 
\begin{linenomath*}
\begin{align*}
   e^{i\theta} \mathbb{I}_N \cdot[F] = e^{i \theta} \mathbb{I}_N \cdot [e^{i\theta} \mathbb{I}_d F] = [e^{i\theta} \mathbb{I}_d F e^{-i \theta}\mathbb{I}_N] = [F],
\end{align*}
\end{linenomath*}
showing that $e^{i \theta} \mathbb{I}_N$ fixes the class $[F]$. Although the descended action is Hamiltonian, the fact that the isotropy subgroup at each point contains a 1-parameter subgroup implies that every point is a singular point of the momentum map. We prefer an effective action so that the momentum map of the $\mathbb{T}$-action has regular values and a full-dimensional image. This choice will help avoid complications later on. 

To get an effective action and avoid the aforementioned complications, we modify the descended torus action on $\tight/U(d)$. For this, we recall the identification of $\tight/U(d)$ with the adjoint orbit ${\mathcal{O}}_{N/d} \subset \H(N) \approx \u(N)^*$ as defined in (\ref{Symplectomorphism}). Under this identification, the descended $\mathbb{T}$-action on $\tight/U(d)$ corresponds to the action on ${\mathcal{O}}_{N/d}$ given by 
\begin{linenomath*}
    \begin{align}
    D \cdot (F^*F) = DF^*FD^*,
    \label{torus action on coadjoint orbit}
    \end{align}
\end{linenomath*}
where $D \in \mathbb{T}$ and $F^*F \in {\mathcal{O}}_{N/d}$. Notice that (\ref{torus action on coadjoint orbit}) is just the adjoint action of $U(N)$ on the orbit ${\mathcal{O}}_{N/d}$, restricted to $\mathbb{T} \leq U(N)$. To obtain an effective action on ${\mathcal{O}}_{N/d}$, we restrict the $\mathbb{T}$-action to the quotient of $\mathbb{T}$ by the center of the group, $Z(U(N)) = \{ e^{i\theta} \mathbb{I}_N\}_{\theta} \approx U(1)$. Specifically, we define
\begin{linenomath*}
    \begin{align*}
        G := \mathbb{T}/Z(U(N)) \approx U(1)^N/U(1) \approx U(1)^{N-1}. 
    \end{align*}
\end{linenomath*}
which we identify with the subgroup of $\mathbb{T}$ consisting of matrices whose $(N,N)$-entry is 1. Since $G$ doesn't contain the 1-parameter subgroup of scalar matrices which act trivially, the toric action of $G$ on ${\mathcal{O}}_{N/d}$ given by (\ref{torus action on coadjoint orbit}) is effective.
\begin{theorem}
The $G$-action on ${\mathcal{O}}_{N/d} \approx \tight/U(d)$ is Hamiltonian with momentum map 
\begin{linenomath*}
    \begin{align*}
        \Phi_G:{\mathcal{O}}_{N/d} \to & \R^{N-1} \approx \g^*\\ 
        F^*F \mapsto & (f_1^*f_1, \dots, f_{N-1}^*f_{N-1}).
    \end{align*}
\end{linenomath*}
\end{theorem}

The proof of this theorem relies on standard results concerning the adjoint action of $U(N)$ on the adjoint orbit ${\mathcal{O}}_{N/d}$, which can be found in \cite{shonk}. Alternatively, one can directly demonstrate $\Phi_G$ is the momentum map by noticing that $\Phi_G$ is the composition
\begin{center}
    \begin{tikzcd}
        {\mathcal{O}}_{N/d} \arrow[r,right hook-latex, "\iota",] & \H(N) \arrow[r, "\alpha"] & \u(N)^* \arrow[r, two heads, "\pi"] & \g \arrow[r, "\phi"] & \R^{N-1}
    \end{tikzcd}
\end{center}
where $\iota$ is the inclusion map, $\alpha$ is the isomorphism defined in equation (\ref{alpha}), $\pi$ is the projection induced by $\iota$, and $\phi$ is the identification given by:
\begin{align*}
   \bigg( \diag(\eta_1, \dots, \eta_{N-1},0) \in \g \mapsto \frac{1}{2}\sum_{k=1}^{N-1}\xi_k\eta_k \bigg) \mapsto (\xi_1, \dots, \xi_{N-1}).
\end{align*}

It is clear that the preimage $\Phi_G^{-1}(\mathbf{1}) \subset \tight/U(d)$, where $\mathbf{1} = (1, \dots, 1) \in \R^{N-1}$, consists of the unitary equivalence classes of unit norm tight frames $F \in \C^{d \times N}$. This is the frame space we were hoping to obtain. However, in order to equip this space with a natural symplectic structure, we must pass to the symplectic quotient. It turns out that $\mathbf{1}$ is not always a regular value of $\Phi_G$, and therefore the symplectic quotient of $\tight/U(d)$ over $\mathbf{1}$ is, in general, a symplectic stratified space (\cref{s-k}). 
\begin{theorem}[{Needham and Shonkwiler~\cite[Proposition~2.16]{shonk}}]
\label{frame space quotient}
For $N>d$ the (non-empty) space 
    \begin{linenomath*}
        \begin{align*}
            \tight(\mathbf{1})/(U(d) \times G) := \big(\tight/U(d)\big)\sslash_{\mathbf{1}}G
        \end{align*}
    \end{linenomath*}
    is a symplectic stratified space. Alternatively, $ \tight(\mathbf{1})/(U(d) \times G) $ can be viewed as the symplectic quotient by the product group
    \begin{linenomath*}
        \begin{align*}
            \C^{d\times N} \sslash_{(-N/d \cdot \mathbb{I}_{d}, \mathbf{1})} (U(d) \times G).
        \end{align*}
    \end{linenomath*}
\end{theorem}

It was shown in \cite{dykema} that $\fun$ is a manifold if and only if $d$ and $N$ are relatively prime. Thus, while $\fun/(U(d) \times G)$ is always a symplectic stratified space, it is a symplectic manifold only when $\text{gcd}(d,N) = 1$. In the next section, we will see that for the purpose of sampling FUNTFs, the fact that $\fun/(U(d)\times G)$ is a symplectic stratified space but not a manifold is not problematic. We have achieved our first goal of obtaining a symplectic structure on (a quotient of) the space of FUNTFs.

Before proceeding to the next section, we briefly pause to compute the dimension of the quotient $\fun/(U(d)\times G)$. To that end, we assume that $\mathbf{1}$ is a regular value of $\Phi_{G}$ so that $\fun/(U(d)\times G)$ is a manifold. We begin with the first reduction. The space $\tight$ is the preimage of the regular value $-N/d \mathbb{I}_d$ under the momentum map $\Phi_{U(d)}$, therefore the regular value theorem implies $\tight$ has codimension $\text{dim }\H(d) = d^2$ in $\C^{d \times N}$. Thus, $\text{dim } \tight = 2dN - d^2$. Moreover, by the quotient manifold theorem 
\begin{linenomath*}
    \begin{align*}
        \text{dim } \tight/U(d) = \text{dim } \tight - \text{dim } U(d) = 2dN - d^2 - d^2 = 2(dN - d^2).
    \end{align*}
\end{linenomath*}
Now, onto the second reduction. The space $\fun/U(d)$ is the preimage of the regular value $\mathbf{1}$ of the momentum map $\Phi_G$, so the codimension of $\fun/U(d)$ in $\tight/U(d)$ is $\text{dim } \mathfrak{g} = N-1$. Hence,
\begin{align*}
    \text{dim } \fun/U(d) = 2(dN - d^2) - (N-1) = 2(dN - d^2) - N + 1.
\end{align*}
Once more, by the quotient manifold theorem 
\begin{linenomath*}
\label{dimension of frame space}
    \begin{align*}
        \text{dim } \fun/(U(d) \times G)  = & \text{dim }\fun/U(d) - \text{dim } G  \\
        = & 2(dN - d^2) - N + 1  - (N-1) = 2(d-1)(N-d-1).
    \end{align*}
\end{linenomath*}
We will refer back to this dimension in the following section.

\section{Toric Geometry of FUNTFs}\label{sec:toric geometry}

We can now refine the statement of our goal, which is to sample FUNTFs. More precisely, the space $\fun$ inherits a Riemannian metric from the standard Euclidean structure on $\C^{d \times N}$, and the associated Riemannian volume measure agrees with the Hausdorff measure $\fun$ inherits as a compact subset of $\C^{d \times N}$. We would like to sample from this measure, but this measure is not so easy to get a handle on. 

In turn, the quotient $\fun/(U(d)\times G)$ (or at least the open dense stratum from \cref{s-k}) has the Riemannian quotient metric with associated Riemannian volume measure. Again, this measure is not so easy to understand using only Riemannian geometry, but the Riemannian volume measure will turn out to agree with the symplectic volume, which is accessible using our symplectic tools.

In particular, we first derive an algorithm for randomly sampling equivalence classes of FUNTFs $[F] \in \fun/(U(d) \times G)$. Afterwards, we can obtain a random FUNTF by simply acting on a representative of the sampled equivalence class by a random element in $U(d) \times G$. 

As noted beneath \cref{frame space quotient}, the quotient  $\fun/(U(d) \times G)$ is a symplectic manifold if and only if $N$ and $d$ are relatively prime. When $\text{gcd}(N,d) \neq 1$, however, the quotient is a symplectic stratified space. By \cref{s-k}, it then contains a unique open, connected, and dense manifold, which we denote by $S$. The complement of $S$ in $\fun/(U(d) \times G)$ has positive codimension, and hence has measure zero with respect to the Riemannian volume measure. Consequently, a random sample from $\fun/(U(d)\times G)$ will almost surely lie in $S$, so it suffices to find an algorithm for sampling the induced volume measure on $S$. Therefore, moving forward, we assume that $\fun/(U(d) \times G)$ is a symplectic manifold, with the implicit understanding that in the case where $\text{gcd}(N,d) \neq 1$, this refers to the unique open, connected, and dense manifold of full measure within $\fun/(U(d) \times G)$.  

Any symplectic manifold $(M,\omega)$ of dimension $2n$ is equipped with a natural measure $m_\omega$, known as the \emph{symplectic} or \emph{Louisville} measure. To construct $m_\omega$, note that the non-degeneracy of $\omega$ ensures the maximal wedge product $\omega^{\wedge n}$ is nowhere vanishing, and thus defines a volume form on $M$. For a Borel set $U \subset M$, the measure is given by $m_\omega(U) := \frac{1}{n!}\int_U \omega^{\wedge n}.$ In the case of $\fun/(U(d) \times G) $, we have the following useful fact: since $\C^{d\times N}$ is Kähler, the symplectic reduction 
\begin{linenomath*}
    \begin{align*}
      \C^{d\times N} \sslash_{(-N/d \cdot \mathbb{I}_{d}, \mathbf{1})} (U(d) \times G) = \fun/(U(d) \times G),
    \end{align*}
\end{linenomath*}
is also Kähler. Moreover, the quotient Riemannian volume measure on $\fun \to \fun/(U(d) \times G)$ coincides with the symplectic measure \cite{nigel}. 

The key to our algorithm is the fact that $\fun/(U(d) \times G)$ admits a half-dimensional Hamiltonian torus action. In particular, the following two powerful theorems apply.
\begin{theorem}[Atiyah~\cite{atiyahConvexityCommutingHamiltonians1982} and Guillemin–Sternberg~\cite{guilleminConvexityPropertiesMoment1982}]
\label{atiyah}
    Suppose $(M, \omega)$ is a (relatively) compact symplectic manifold of dimension $2n$ that admits a Hamiltonian action of a torus $U(1)^n$.  Let $\Phi : M \to (\u(1)^n)^* \approx \R^n$ be the momentum map of the torus action. Then 
    \begin{itemize}
        \item For any $v \in \R^n$, $\Phi^{-1}(v)$ is either empty or connected.
        \item (The closure of) the image of $M$ is the convex hull of the images of the fixed point of the torus action.
    \end{itemize}
\end{theorem}
A $2n$-dimensional symplectic manifold $(M,\omega)$ is called a \emph{toric symplectic manifold} if it admits a Hamiltonian action by the torus $U(1)^n$. By \cref{atiyah}, (the closure of) the image $\Phi(M)$ is a convex polytope $P$, known as the \emph{moment polytope} of the torus action. 

 We now describe a process for generating random points on toric symplectic manifolds which are also Kähler, which we call \emph{toric Kähler manifolds}. This process will later be applied to the specific manifold $\fun/(U(d) \times G)$. 
 
 Let $M$ be a $2n$-dimensional toric symplectic manifold that admits a Hamiltonian toric action $\phi: U(1)^n \to \operatorname{Diff}(M)$, with momentum map $\Phi: M \to (\u(1)^n)^* \approx \R^n$ and moment polytope $P$. Define $\Tilde{M}$ as the open, dense, full measure submanifold consisting of regular points of $\Phi$. Note that $\Phi(\Tilde{M})$ has full measure in $P$. The restriction, $\Phi|_{\Tilde{M}}: \Tilde{M} \to \Phi( \Tilde{M})$, forms a $U(1)^n$-fiber bundle \cite{da_Silva}. 
 
\begin{definition}\label{def:rough section}
	We will call a map $\alpha : \Phi(\Tilde{M}) \to \Tilde{M}$ a \emph{rough section} of $\Phi$ if $\Phi \circ \alpha$ is the identity map on $\Phi(\Tilde{M})$. Note that there is no assumption that $\alpha$ is smooth or even continuous; hence the modifier ``rough.''
\end{definition}
 
 Suppose there exists a rough section $\alpha : \Phi(\Tilde{M}) \to \Tilde{M}$. Let $\lambda$ denote the Lebesgue measure on $P$, and let $\nu$ denote the standard product measure on $U(1)^n$. Consider the following sampling scheme to generate points on $\Tilde{M}$. 
 \begin{linenomath*}
\begin{algorithm}
\caption{Toric Kähler sampling}\label{alg:samp}
\begin{algorithmic}
\State $x \sim \text{Uniform}(\Phi(\Tilde{M}), \lambda)$
\State $y \gets \alpha(x) \in \Phi^{-1}(x)$
\State $\theta \sim \text{Uniform}(U(1)^n, \nu)$
\State $p \gets \phi(\theta, y) \in \Tilde{M}$
\State \Return $p$
\end{algorithmic}
\end{algorithm}
 \end{linenomath*}
 
To summarize, we begin by choosing a random point $x$ in the image of $\Tilde{M}$ under $\Phi$. Using $\alpha$, we determine a specific $y$ in the fiber of $\Phi$ above $x$; this process can be viewed as selecting a random fiber. Next, we choose a random point $p$ on this fiber by applying a randomly selected torus element $\theta$ to $y$. 
\begin{theorem}
    \label{sampling uniform}
   Points generated by \cref{alg:samp} are sampled from the symplectic measure $m_\omega$ on $M$. 
\end{theorem}

While this theorem essentially follows from the Duistermaat--Heckman theorem~\cite{duistermaatVariationCohomologySymplectic1982} (see~\cite[Chapter~30]{da_Silva} or~\cite[Theorem~1]{cantarellaSymplecticGeometryClosed2016}), we are not aware of an explicit proof in the literature, so we provide one here.

\begin{proof}
    We begin by equipping a fiber of the momentum map $\Phi$ with coordinates. For any $x \in \Phi(\Tilde{M})$, define $M_x := \Phi^{-1}(x)$. Notice that $\Phi^{-1}(x) \approx U(1)^n$ through the mapping
    \begin{linenomath*}
        \begin{align*}
           \psi_x := \phi(\, \cdot\, , \alpha(x)) : U(1)^n \to M_x, \hspace{15pt} \theta \mapsto \phi (\theta, \alpha(x)),
        \end{align*}
    \end{linenomath*}
    where $\phi$ is the torus action. Under this identification, the standard angle coordinates $(\theta^1, \dots, \theta^n)$ in $U(1)^n$ map to coordinates $(\Tilde{\theta}_x^1, \dots, \Tilde{\theta}_x^n)$ on $M_x$ given by $\Tilde{\theta}^i_x = \theta^i \circ \psi_x^{-1}$. 

    Next, we derive the Riemannian volume form $dM_x$ in terms of the coordinates $(\Tilde{\theta}_x^i)$. For $i = 1, \dots, n$ let $Y_i \in \mathfrak{X}(\Tilde{M})$ be the infinitesimal vector field generated by $\frac{\partial }{\partial x^i}\big|_{x} \in T_x\R^n$ and $\Phi^i = x^i \circ \Phi$, where $(x^1, \dots, x^n)$ are the standard coordinates on $\R^n$. By definition, $Y_i$ is the Hamiltonian vector field of the circle action $\phi^i$ and thus $Y_i|_p = \frac{\partial }{ \partial \Tilde{\theta}_x^i}|_p$. For any $Z \in \mathfrak{X}(M)$, 
    \begin{align*}
        d\Phi^i(Z) = \omega(Y_i,Z) = g(Y_i, JZ) = g(-JY_i,Z)
    \end{align*}
    where $g$ is the Riemannian metric and $J$ is the complex structure on $M$. 
	
	On the other hand, by definition of the Riemannian gradient we know $d\Phi^i(Z)= g(\nabla \Phi^i, Z)$. Thus, for every $Z \in \mathfrak{X}(M)$, $g(-JY_i, Z) = g(\nabla \Phi^i, Z)$. By the non-degeneracy of $g$, we have $-JY_i = \nabla \Phi^i$. This equality allows us to make the following observation:
    \begin{linenomath*}
    \begin{align*}
         |g| = &\sqrt{\det \bigg[ g\bigg( \frac{\partial}{\partial \Tilde{\theta}_x^i}, \frac{\partial}{\partial \Tilde{\theta}_x^j}\bigg) \bigg]_{i,j}} = \sqrt{\det [g(Y_i, Y_j)]_{i,j}} 
        =\sqrt{\det [g(-JY_i, -JY_j)]_{i,j}} \\& = \sqrt{\det [g(\nabla \Phi^i, \nabla \Phi^j)]_{i,j}} = \sqrt{\det(D\Phi)(D\Phi)^T} = NJ \Phi   ,
    \end{align*}   
    \end{linenomath*}
    where $NJ\Phi$ is the normal Jacobian of $\Phi$. Thus,
    \begin{linenomath*}
        \begin{align}
        \label{volumeform}
            dM_x = NJ \Phi\, d \tilde{\theta}_x^1 \wedge \cdots \wedge d\tilde{\theta}_x^n.
        \end{align}
    \end{linenomath*}

    Now, we describe the measure from which \cref{alg:samp} samples. For a Borel set $U \subset M$, define ${U_x := U \cap M_x}$. Consider the measure $\mu_x$ on the fiber $M_x$, defined as
    \begin{linenomath*}
        \begin{align*}
        \mu_x(U_x) := \int_{\theta \in T^n} \phi(\theta,\, \cdot \,)_* \delta_{\alpha(x)}(U_x) d\nu(\theta) ,
    \end{align*}
    \end{linenomath*}
    where $\phi( \theta, \, \cdot \,): M_x \to M_x$ is the map $p \mapsto \phi( \theta, p)$ and $\delta_{\alpha(x)}$ is the Dirac measure on $M_x$ centered at $\alpha(x)$. Notice that the measure $\mu_x$ can be sampled by applying a random torus element $\theta \sim \text{Uniform}(U(1)^n, \nu)$ to $\alpha(x)$. It follows that \cref{alg:samp} generates samples from the measure $\mu$ on $\Tilde{M}$ defined as 
    \begin{linenomath*}
        \begin{align*}
            \mu(U) := \int_{x \in \Phi(U)} \mu_x(U_x) d\lambda(x).
        \end{align*}
    \end{linenomath*}
    Thus, our goal is to show that $m_\omega = \mu$.

    To better understand the measure $\mu$ from which \cref{alg:samp} samples, observe that 
    \begin{linenomath*}
        \begin{align*}
            \phi(\theta, \, \cdot \,)_* \delta_{\alpha(x)}(U_x) = & \begin{cases}
                1 & \alpha(x) \in \phi(\theta, \, \cdot\, )^{-1}(U_x)\\
                0 & \text{otherwise}
            \end{cases}\\
            = & \begin{cases}
                1 & \phi(\theta, \alpha(x)) \in U_x \\
                0 & \text{otherwise}.
            \end{cases}
        \end{align*}
    \end{linenomath*}
This is equivalent to the characteristic function $\chi_{\psi_x^{-1}(U_x)}$. Thus, 
\begin{linenomath*}
    \begin{align*}
    \mu_x(U_x) = \int_{\theta \in U(1)^n} \chi_{\psi_x^{-1}(U_x)}(\theta) d\nu(\theta) = (\psi_x)_* \nu(U_x).
\end{align*}
\end{linenomath*}
Since the standard product measure $\nu$ coincides with the Riemannian volume form $dU(1)^n = d\theta^1 \wedge \cdots \wedge d\theta^n$ on $U(1)^n$, we can use a change of coordinates to express this as: 
\begin{linenomath*}
    \begin{align*}
    (\psi_x)_* \nu(U_x) = & \int_{p \in M_x} \chi_{U_x}(p) d(\psi_x)_* \nu(p) \\
    = & \int_{\theta \in \psi_x^{-1}(U_x)} |\det D \psi_x (\theta) | dU(1)^n,
\end{align*}
\end{linenomath*}
where $|\det D \psi_x (\theta)|$ is the determinant of the Jacobian of $\psi_x$ evaluated at $\theta$ with respect to the the angle coordinates $(\theta^i)$ on $U(1)^n$ and the coordinates $(\Tilde{\theta}^i_x)$ on $M_x$. Since $\Tilde{\theta}^i_x = \theta^i \circ \psi_x^{-1}$, it follows that $D \psi_x(\theta) = \text{Id}_n$. 

Using this simplification in combination with the fact that Lebesgue measure on $\Phi(\Tilde{M})$ coincides with its Riemannian volume form, we can write $\mu$ as 
\begin{linenomath*}
    \begin{align}
    \label{simplified mu}
        \mu(U) = \int_{x \in \Phi(U)} \bigg( \int_{\theta \in \psi_x^{-1}(U_x)} 1 dU(1)^n \bigg) d\Phi(\Tilde{M}).
    \end{align}
\end{linenomath*}

Finally, we turn our attention to the symplectic measure $m_\omega$ and show it agrees with \eqref{simplified mu}. Because $\Tilde{M}$ is Kähler, $\Tilde{M}$ is also a Riemannian manifold, and its Riemannian volume form coincides with the symplectic measure. Applying the smooth coarea formula and recalling that $dM_x  = NJ\Phi\, d\Tilde{\theta}_x^1 \wedge \cdots \wedge d\Tilde{\theta}_x^n$ from \eqref{volumeform}, we have 
\begin{linenomath*}
    \begin{multline*}
    m_\omega(U) =  \int_{M} \chi_U d\Tilde{M} =  \int_{\Phi(U)} \bigg( \int_{U_x} \frac{1}{NJ \Phi}dM_x \bigg) d\Phi(\Tilde{M}) =  \int_{\Phi(U)} \bigg( \int_{U_x} d\Tilde{\theta}_x^1 \wedge \cdots \wedge d\Tilde{\theta}_x^n \bigg) d\Phi(\Tilde{M}) \\
    =  \int_{\Phi(U)} \bigg( \int_{\psi_x^{-1}(U_x)} d\theta^1 \wedge \cdots \wedge d\theta^n \bigg) d\Phi(\Tilde{M}) = \int_{x \in \Phi(U)} \bigg( \int_{\theta \in \psi_x^{-1}(U_x)} dU(1)^n \bigg) d\Phi(\Tilde{M}) =  \mu(U)
    \end{multline*}
    as desired.
\end{linenomath*}
\end{proof}

\subsection{Circle Actions}

We define a Hamiltonian torus action on $\fun/(U(d) \times G)$ as a product of individual circle actions, which we now describe. First, we define the $U(1)$-actions on $\mathcal{F}^{d,N}$ and then descend these actions to the quotient $\fun/(U(d) \times G)$. 

Let $F = [f_1\,|\,\cdots\, |\, f_N] \in \framespace$, and let $k \in \{1, \dots, N\}$. Recall that the $k^{th}$ partial frame operator of $F$ is defined  as $S_k := f_1f_1^* + \dots + f_kf_k^*$. The matrix $S_k$ belongs to $\H(d)$ and has rank at most $\text{min}\{k,d\}$. In particular, $S_k$ is unitarily diagonalizable, meaning there exist $U_k = [u_{k,1} \, |\, \cdots \, | \, u_{k,d}] \in U(d)$ and $\mu_{k,1} \geq \cdots \geq \mu_{k,d} \in \R$ such that $S_k = U_k^* \diag(\mu_{k,1}, \dots, \mu_{k,d})U_k$, where $u_{k,j}$ is the eigenvector of $S_k$ corresponding to the eigenvalue $\mu_{k,j}$. As introduced in \cref{frame prelim}, the values $\mu_{k,j}$ are called the \emph{eigensteps} of the frame $F$. Now, fix $j \in \{ 1, \dots, \min\{k,d\}\}$ and assume that $\mu_{k,j}$ is an isolated eigenvalue of $S_k$. This assumption implies that the eigenspace of $\mu_{k,j}$ is $\text{span}_\C \{u_{k,j}\}$, so there is only one orthonormal basis for this eigenspace, up to the choice of scalar factor. Since $u_{k,j}u_{k,j}^* \in \H(d)$, it follows that $\i u_{k,j}u_{k,j}^* \in \u(d)$ and hence $\{\exp(t\i u_{k,j}u_{k,j}^*)\}_{t \in \R}$ is the 1-parameter subgroup of $U(d)$ generated by $\i u_{k,j}u_{k,j}^*$. We define the $U(1)$-action as follows: 
\begin{linenomath*}
    \begin{align}
    \phi_{k,j} :U(1) \times \framespace \to & \framespace \\ 
        (t,F) \mapsto & \bigg[ \text{exp}(t\i u_{k,j}u_{k,j}^*) \big[f_1\, | \, \cdots \, | \, f_k \big] \, | \,  f_{k + 1} \, | \, \cdots \, | \, f_N\bigg], \label{phikj}
    \end{align}
\end{linenomath*}
so that $\phi_{k,j}$ multiplies the first $k$ frame vectors of $F$ by $\exp(t \i u_{k,j}u_{k,j}^*)$ while leaving the remaining $N-k$ frame vectors fixed. Note that $u_{k,j}$ depends on $F$.
\begin{remark}
The purpose of this remark is to clarify why $\phi_{k,j}$ is a $U(1)$-action. At first glance, $\phi_{k,j}$ may appear define an action by $\R$, since we could evaluate \eqref{phikj} for any $t \in \R$. However, $\phi_{k,j}(t, F) = \phi_{k,j}(t + 2\pi n, F)$ for any $t \in \R$ and $n \in \Z$; that is, $\phi_{k,j}$ is $2\pi$-periodic. To see this, note that $u_{k,j}u_{k,j}^* \in \H(d)$ is a rank-1 projector, so $(u_{k,j}u_{k,j}^*)^2 = u_{k,j}u_{k,j}^*$. This allows us to simplify the Taylor series expansion of the matrix exponential as follows:
\begin{linenomath*}
\begin{align*}
    \text{exp}(t \i u_{k,j} u_{k,j}^*) & =  \mathbb{I}_d + \sum_{m = 1}^\infty \frac{(t\i)^m (u_{k,j}u_{k,j}^*)^m}{m!} \\ & = \mathbb{I}_d + u_{k,j}u_{k,j}^*\sum_{m = 1}^\infty \frac{(t \i)^m}{m!} \\ & =  \mathbb{I}_d + u_{k,j}u_{k,j}^*(e^{t\i} - 1),
\end{align*}
\end{linenomath*}
which shows the action is indeed $2 \pi$-periodic. Thus, taking values $t \in [0,2\pi)$ captures all of $\phi_{k,j}$, and identifying $[0,2\pi)$ with $U(1)$ via the mapping $t \mapsto e^{t\i}$ confirms that $\phi_{k,j}$ is a genuine $U(1)$-action on $\framespace$. 
\end{remark}
\begin{remark}
    \label{isolated}
   Technically, the action $\phi_{k,j}$ is only defined on the open dense subset of $\framespace$ consisting of frames for which the $k^{\text{th}}$ partial frame operator has its $j^{\text{th}}$ eigenvalue isolated when the spectrum is arranged in nonincreasing order. Specifically, if $F$ is such that $\mu_{k,j}$ is an eigenvalue of $S_k$ with multiplicity $\ell > 1$, then there is no unique choice of an eigenvector $u_{k,j}$ corresponding to $\mu_{k,j}$.
\end{remark}

\begin{theorem}[{Needham and Shonkwiler~\cite[Proposition~3.13]{shonk}}]
The circle action $\phi_{k,j}$ on $\framespace$ is Hamiltonian and the momentum map of the action is $\Phi_{k,j}$ given by
    \begin{linenomath*}
        \begin{align*}
            \Phi_{k,j}(F) = \mu_{k,j}
        \end{align*}
    \end{linenomath*}
    defined on the open and dense submanifold of $\framespace$ consisting of frames whose $k^{th}$ partial frame operator have isolated $j^{th}$ eigenvalue (arranged in descending order). As above, $\mu_{k,j}=\mu_{k,j}(F)$ is the $j^{th}$ eigenvalue of the $k^{th}$ partial frame operator of $F$. 
\end{theorem}

\subsection{Descended Circle Actions and Toric Structure on \texorpdfstring{$\fun/(U(d) \times G)$}{FUNTF space}}
We pause briefly to introduce some notation. In \cref{isolated}, we noted that the circle actions are defined only on a specific dense subset of $\framespace$. To track the subsets where these actions are defined, let $\mathcal{I}$ denote the index set containing all pairs $(k,j)$ for which $\mu_{k,j}$ is an independent eigenstep of a generic FUNTF. Specifically, $(k,j) \in \mathcal{I}$ if and only if $1\leq j \leq d-1$ and $j+1 \leq k \leq j + N-d-1$. The part of the table in \cref{fig:eigensteps} that corresponds to these independent eigensteps is shown in \cref{fig:ind steps}.  
\begin{figure}[h]
    \centering
    \begin{center}
    \begin{tikzcd}[column sep=tiny, row sep=tiny]
        &   &  &   & \mu_{N-2,d-1} \\
        &   &  & \mu_{N-3,d-2}  & \mu_{N-3,d-1} \\
        &   & \iddots &  \vdots & \vdots \\
         & \mu_{N-d + 1, 2} & \cdots & \mu_{N-d + 1, d-2} & \mu_{N-d+1, d-1}\\
       \mu_{N-d,1} &  \mu_{N-d,2} & \cdots  &  \mu_{N-d,d-2} & \mu_{N-d,d-1} \\
       \mu_{N-d-1, 1} & \mu_{N-d-1,2} & \cdots & \mu_{N-d-1,d-2} & \\
       \vdots & \vdots & \iddots & \\ 
        \mu_{3,1}&\mu_{3,2}  \\
        \mu_{2,1}  \\

    \end{tikzcd}
    \end{center}
    \caption{The independent eigensteps of a generic FUNTF in $\fun$. Once these values are determined, the full eigenstep table (\cref{fig:eigensteps}) can be constructed by setting the upper left triangle of eigensteps equal to $N/d$ and choosing $\mu_{k,\min\{k,d\}}$ to be the unique nonnegative value which satisfies the constraint $\sum_{j = 1}^{\min\{k,d\}} \mu_{k,j} = k$.}
    \label{fig:ind steps}
\end{figure}

We use $(\mu_{k,j})_{(k,j) \in \mathcal{I}}$ to denote the vector in $\R^{|\mathcal{I}|}$ formed by stacking the columns of the table in \cref{fig:ind steps}. For $(k,j) \in \mathcal{I}$, let $\mathcal{U}_{k,j}$ denote the open and dense subset of $\fun$ consisting of frames $F$ such that $\mu_{k,j}(F)$ is an isolated eigenvalue of $S_k(F)$. Then $q(\mathcal{U}_{k,j})$---where $q : \fun \to \fun/(U(d) \times G)$ is the quotient map---is an open and dense subset of $\fun/(U(d)\times G)$. Finally, let 
\begin{linenomath*}
    \begin{align*}
        \mathcal{U}_\mathcal{I} := \bigcap_{(k,j) \in \mathcal{I}} \mathcal{U}_{k,j}
    \end{align*}
\end{linenomath*}
which is open and dense in $\fun$, making $q(\mathcal{U}_\mathcal{I})$ open and dense in $\fun/(U(d)\times G)$ as well, so it has full measure with respect to the Riemannian quotient measure. 

It was shown in \cite{shonk} that the circle actions $\phi_{k,j}$ preserve both the frame vector norms and the frame operator, so that $\phi_{k,j}$ restricts to a $U(1)$-action on $\mathcal{U}_{k,j} \subset \fun$. Moreover, it was shown that $\phi_{k,j}$ commutes with the left $U(d)$-action and the right $G$-action i.e., $A\phi_{k,j}(t,F)D = \phi_{k,j}(t,AFD)$ for each $A \in U(d)$ and $D \in G$. Consequently, $\phi_{k,j}$ descends to a well-defined action $\Tilde{\phi}_{k,j}$ on the open, dense, full measure subset $q(\mathcal{U}_{k,j}) \subset \fun/(U(d) \times G)$ given by $\Tilde{\phi}_{k,j}(t,[F]) = [\phi_{k,j}(t,F)]$.

When $(k,j) \in \mathcal{I}$, the momentum map $\Phi_{k,j}$ descends to a well-defined map on $q(\mathcal{U}_{k,j}) \subset \fun/(U(d) \times G)$. To see this, take an $F \in \mathcal{U}_{k,j}$ and let $AFD$ be an element in the $U(d) \times G$ orbit of $F$, where $A \in U(d)$ and $D \in G$. The $k^{th}$ partial frame operators of $F$ and $AFD$ are given by 
\begin{linenomath*}
    \begin{align*}
        S_k(F) = f_1f_1^* + \cdots + f_kf_k^*
    \end{align*}
\end{linenomath*}
and 
\begin{linenomath*}
    \begin{align*}
        S_k(AFD)  =  Af_1D(Af_1D)^* + \cdots + Af_kD(Af_kD)^* = Af_1f_1^*A^* + \cdots + Af_kf_k^*A^*.
    \end{align*}
\end{linenomath*}
Notice that $S_k(AFD)$ is simply $S_k(F)$ conjugated by the unitary matrix $A$, so $S_k(F)$ and $S_k(AFD)$ have identical spectra. This implies that $\Phi_{k,j}$ is $(U(d) \times G)$-invariant, so the map 
\begin{linenomath*}
    \begin{align*}
        \Tilde{\Phi}_{k,j} : q(\mathcal{U}_{k,j}) \to  \R \approx \u(1)^*, \hspace{15pt}
        [F] \mapsto \Phi_{k,j}(F)
    \end{align*}
\end{linenomath*}
is well-defined. Furthermore, $\Tilde{\Phi}_{k,j}$ serves as momentum map for the descended action, making $\Tilde{\phi}_{k,j}$ a Hamiltonian action.
\begin{theorem}
    For each $(k,j) \in \mathcal{I}$, the $U(1)$-action $\phi_{k,j}$ descends to a well defined circle action $\Tilde{\phi}_{k,j}$ on the open and dense full measure submanifold $q(\mathcal{U}_{k,j})$ of $\fun/(U(d) \times G)$. Moreover, the action is Hamiltonian and $\Tilde{\Phi}_{k,j}$ is the momentum map.
    \label{circleacts}
\end{theorem}
\begin{proof}
This follows from a more general fact about Hamiltonian actions. If the product Lie group $H_1 \times H_2$ acts on a symplectic manifold $(M,\omega)$ in a Hamiltonian way with momentum map $(\mu_1, \mu_2): M \to \mathfrak{h}_1^* \times \mathfrak{h}_2^*$, then $H_1$ and $H_2$ individually act on $M$ and the two actions commute. Moreover, for any coadjoint orbit $\mathcal{O} \subset \mathfrak{h}_1^*$, the $H_2$ action is $H_1$-invariant on $M\sslash_\mathcal{O} H_1$ so the $H_2$-action descends to a well defined action on the quotient $M\sslash_{\mathcal{O}} H_1$. The descended action is Hamiltonian with momentum map $\Tilde{\mu}_2$ characterized by $q^* \Tilde{\mu}_2 = \iota^* \mu_2$ where $q: \mu_1^{-1}(\mathcal{O}) \to M\sslash_\mathcal{O} H_1$ is the quotient map and $\iota: \mu_1^{-1}(\mathcal{O}) \to M$ is the inclusion map. Then setting $M = \framespace$, $H_1 = U(d)\times G$, and $H_2 = U(1)$ gives the desired result. 
\end{proof}
\begin{theorem}
 The torus $U(1)^{d_T}$ of dimension $d_T := (d-1)(N-d-1)$ acts on the open, dense, and full measure submanifold $q(\mathcal{U}_{\mathcal{I}}) \subset \fun/(U(d) \times G)$ as the product of the circle actions $\phi = \bigtimes_{(k,j) \in \mathcal{I}}{\Tilde{\phi}_{k,j}}$. The action is Hamiltonian and $\Phi: \mathcal{U}_\mathcal{I} \to \R^{d_T} \approx (\u(1)^{d_T})^*$ given by $[F] \mapsto (\mu_{k,j}(F))_{(k,j)\in \mathcal{I}}$ is the momentum map. Since
 \begin{linenomath*}
     \begin{align*}
         d_T = \frac{1}{2} \text{ dim } \fun / (U(d) \times G) = \frac{1}{2} \text{ dim } \mathcal{U}_\mathcal{I}, 
     \end{align*}
 \end{linenomath*} 
$\mathcal{U}_{\mathcal{I}}$ is a toric symplectic manifold.
\end{theorem}
\begin{proof}
    As shown in \cite{shonk}, for each $(k,j), (l,m) \in \mathcal{I}$ the circle actions $\Tilde{\phi}_{kj}$ and $\Tilde{\phi}_{kj}$ commute; i.e., 
    \begin{linenomath*}
    \begin{align*}
        \Tilde{\phi}_{kj}(t, \Tilde{\phi}_{lm}(s,[F])) = \Tilde{\phi}_{lm}(s, \Tilde{\phi}_{kj}(t,[F])).
    \end{align*}
    \end{linenomath*}
    Additionally, it is easy to see that the cardinality of $\mathcal{I}$ is $d_T$. 
\end{proof}
Let $\mathcal{P}_{d,N} := \Phi(\mathcal{U}_\mathcal{I})$ be the moment polytope of the Hamiltonian toric action on $\mathcal{U}_\mathcal{I}$. Then $\overline{\mathcal{P}}_{d,N} := \text{cl }\mathcal{P}_{d,N}$ is the convex polytope in $\R^{d_T}$ consisting of all vectors of independent eigensteps of FUNTFs; i.e., 
\begin{linenomath*}
    \begin{align*}
        \overline{\mathcal{P}}_{d,N} = \{ (\mu_{k,j}(F))_{(k,j) \in \mathcal{I}} : F \in \fun \}.
    \end{align*}
\end{linenomath*}
This polytope is defined by Weyl's inequalities and the row sum constraints, as discussed in \cref{fig:eigensteps,fig:ind steps}. For general $d$ and $N$, there is no straightforward way to explicitly list all the defining inequalities for $\overline{\mathcal{P}}_{d,N}$; however, they can be generated algorithmically with relative ease. Our proof-of-concept implementation, available on GitHub~\cite{sampler}, demonstrates this process. The defining inequalities of $\mathcal{P}_{3,5}$ were explicitly worked out in \cref{35example}.

Notice, given an $[F] \in \mathcal{U}_\mathcal{I}$ we can reconstruct the complete eigenstep table of $F$ from $\Phi([F])$ using the fact that $\sum_{j = 1}^{\text{min}\{k,d\}} \mu_{k,j}(F) = k$ and  $\mu_{k,j}(F) = N/d$ whenever $k = N-d + 1, \dots, N$ and $j = 1, \dots, k-(N -d)$. For example, if $F \in \mathcal{F}^{3,5}_{5/3}(1)$ and $\Phi([F]) = (x_1,x_2) \in \overline{\mathcal{P}}_{3,5}$ then the complete eigenstep table of $F$ can be constructed by plugging $x_1$ and $x_2$ into the table of Example $\ref{35example}$. 

We have established that \( \fun/(U(d) \times G) \) is a toric Kähler manifold that admits the Hamiltonian action \( \phi \) with moment map \( \Phi \) and moment polytope \( \mathcal{P}_{d,N} \). Furthermore, following the procedure described in Cahill et al.~\cite[Theorem~7]{fickus}—choosing the identity matrix whenever a choice arises in their algorithm—we define the function \verb|lift_to_fiber|. This function maps a point \( \mu \in \mathcal{P}_{d,N} \) to a representative \( F \) of the unitary equivalence class \( [F] \in \fun/(U(d) \times G) \) such that \( \Tilde{\Phi}([F]) = \mu \). In other words, \verb|lift_to_fiber| serves as a rough section of the \( U(1)^n \) bundle \( \fun/(U(d) \times G) \). Using these components in \cref{alg:samp} yields our \eigenlift algorithm, stated below as \cref{alg:sampling FUNTF}. 

    \begin{algorithm}
        \caption{Eigenlift}\label{alg:sampling FUNTF}
    \begin{algorithmic}[1]
        \Require $d + 1 < N$
        \State $ \mu \sim \text{Uniform}(\overline{\mathcal{P}}_{d,N}, \lambda)$
        \State $F \gets \verb|lift_to_fiber|(\mu)$ defined using Theorem 7 in \cite{fickus}
        \State $\theta \sim \text{Uniform}(U(1)^n, \nu)$
        \State $F_{new} \gets \phi(\theta,F)$
        \State \Return $[F_{new}]$
    \end{algorithmic}
    \end{algorithm}

\begin{corollary}
    For any $d, N \in \N$ with $d + 1  < N$, \cref{alg:sampling FUNTF} produces a uniformly random $[F] \in \fun/(U(d) \times G)$ with respect to the Riemannian quotient volume measure. Moreover, the representative $F$ is full-spark with probability 1. 
\end{corollary}
    
\begin{proof}
    The first sentence follows from applying \cref{sampling uniform}, and the second sentence is a consequence of \cref{full-spark}. 
\end{proof}

\section{Implementation}\label{sec:implementation}
As a proof of concept, we implemented this algorithm in Python. The implementation supports any values of $d$ and $N$, provided that $N > d + 1$. However, for large $d$ and $N$, our simple implementation becomes intractable due to computational complexity. It is well known in the literature that uniformly sampling a convex polytope is a challenging problem, and our implementation reflects this difficulty.  

The computational complexity of \cref{alg:sampling FUNTF} is primarily dictated by the first step: sampling $\mu$ from $\overline{\mathcal{P}}_{d,N}$. We attempted to use the standard hit-and-run algorithm~\cite{hnr,bonehConstraintsRedundancyFeasible1979,andersenHitRunUnifying2007,smithEfficientMonteCarlo1984}, which is a Markov chain that converges to Lebesgue measure on any convex polytope. However, it performed poorly, exhibiting extremely slow convergence on our polytope $\overline{\mathcal{P}}_{d,N}$. 

For the purpose of this paper, our goal is to demonstrate a proof of concept in low dimensions. Given this, a reasonable approach is to use rejection sampling within the bounding box of $\overline{\mathcal{P}}_{d,N}$. However, for the algorithm to be practical with larger values of $d$ and $N$, a new sampling paradigm must be developed for generating random independent eigensteps. Note that the Top Kill algorithm~\cite{fickusConstructingAllSelfadjoint2013} samples eigensteps from a measure on $\overline{\mathcal{P}}_{d,N}$ which is absolutely continuous with respect to Lebesgue measure, but not from Lebesgues masure itself.

In our proof-of-concept experiment, we sampled \(10^6\) uniformly random unitary classes of FUNTFs from \(\mathcal{F}_N / (U(d) \times G)\) using Algorithm~\ref{alg:sampling FUNTF}. We then recorded the coherence, defined for a unit norm frame  $F$ as \( \text{Coh}(F) := \max_{i \neq j} |\langle f_i, f_j \rangle|\), of the representative in a histogram. The results are shown in \cref{fig:coherences}.

\begin{figure}[H]
    \centering
    \adjustbox{max width=\textwidth}{%
    \begin{tabular}{ccc}
        \includegraphics[width=0.3\textwidth]{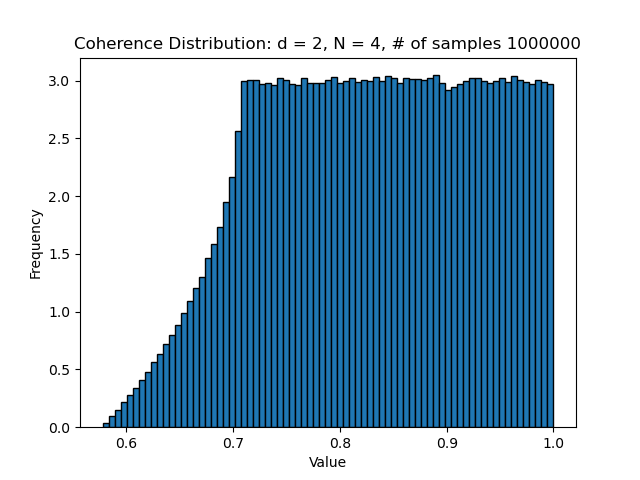} &
        \includegraphics[width=0.3\textwidth]{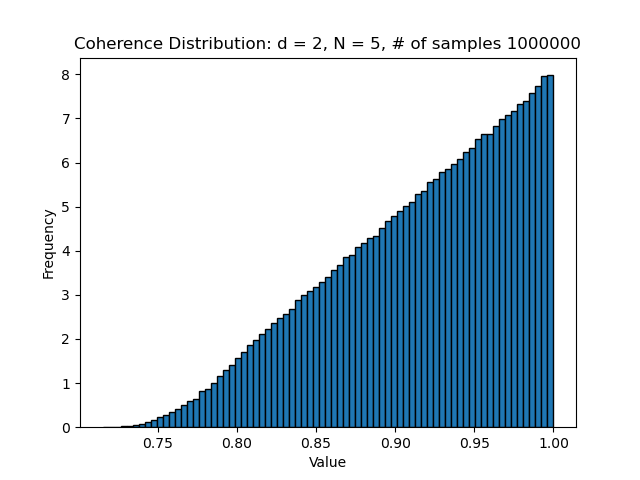} &
        \includegraphics[width=0.3\textwidth]{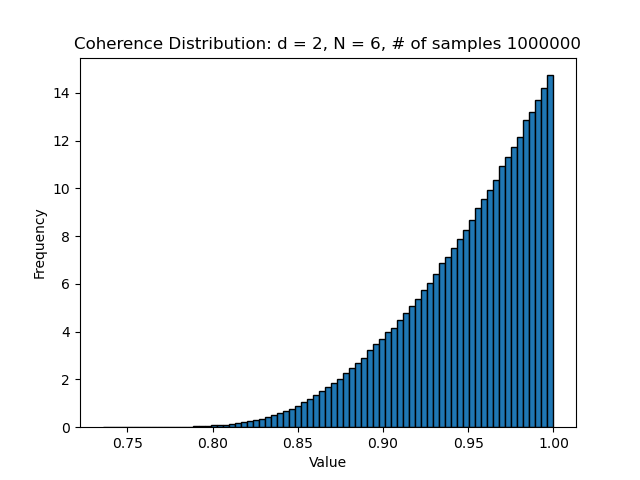} \\
        \includegraphics[width=0.3\textwidth]{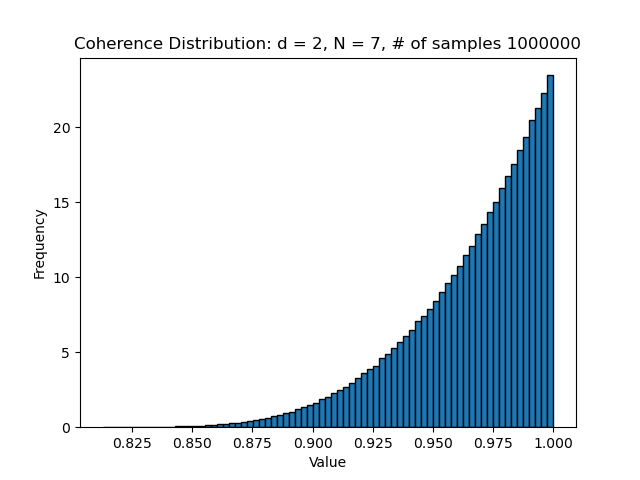} &
        \includegraphics[width=0.3\textwidth]{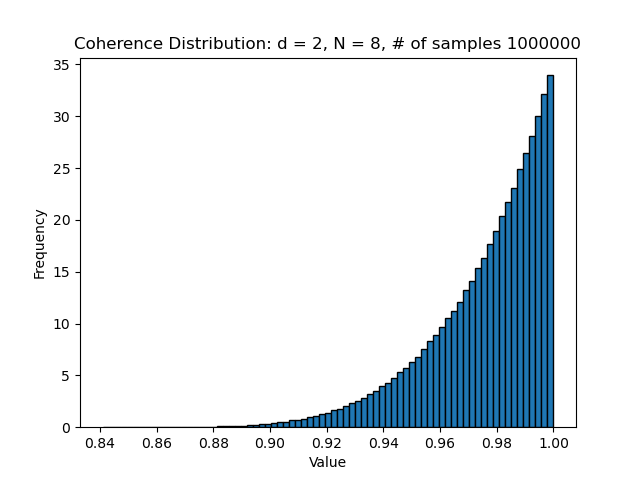} &
        \includegraphics[width=0.3\textwidth]{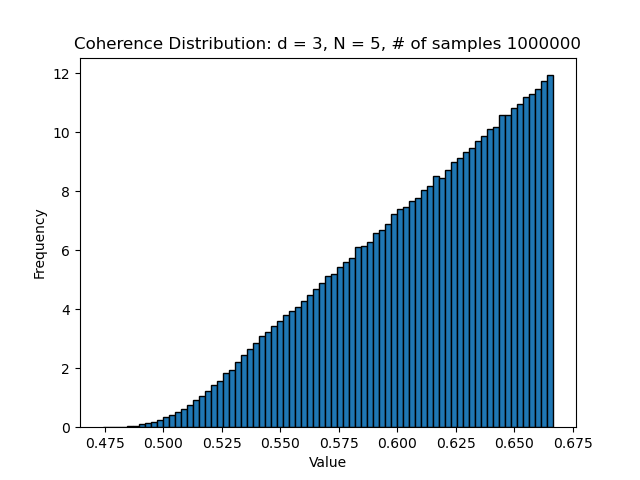} \\
        \includegraphics[width=0.3\textwidth]{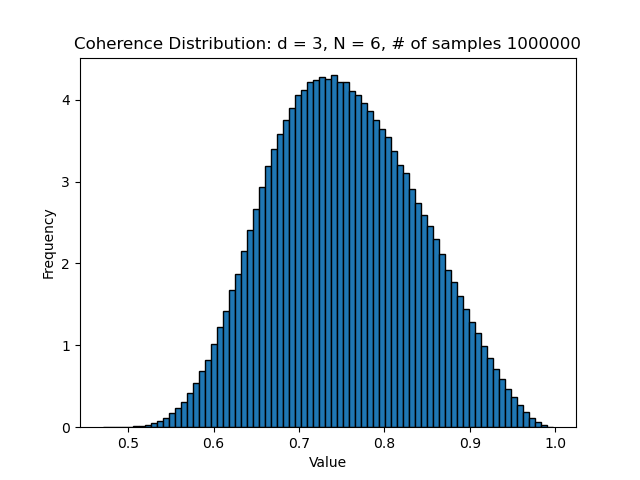} & 
        \includegraphics[width=0.3\textwidth]{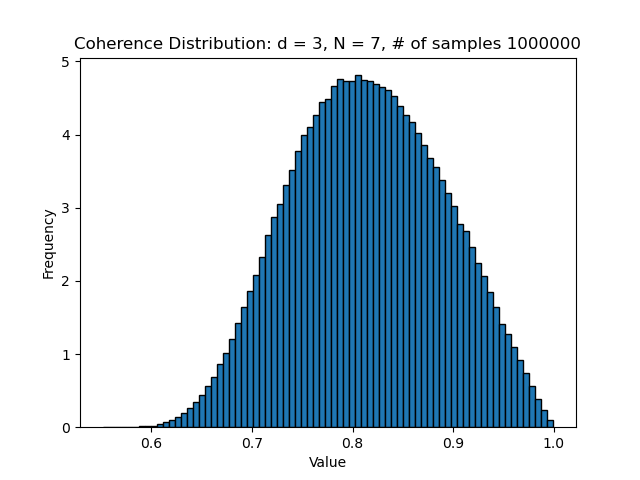} &
        \includegraphics[width=0.3\textwidth]{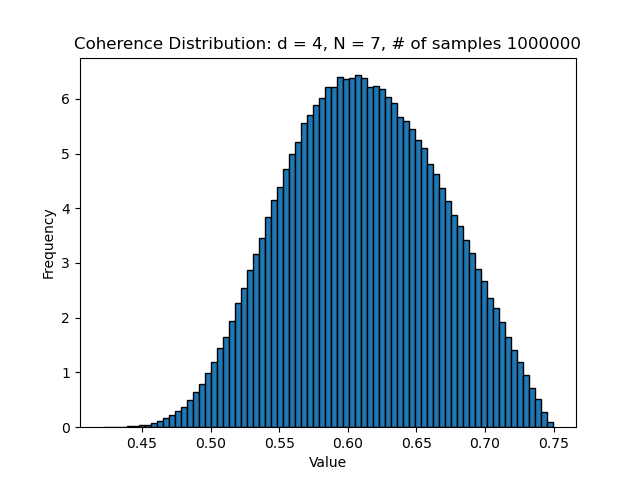}
    \end{tabular}
    }
    \caption{Distribution of coherences of representatives of random unitary classes FUNTF's generated using Algorithm~\ref{alg:sampling FUNTF}.}
    \label{fig:coherences}
\end{figure}

We note that these coherence values lie within the known bounds listed in \cite{sloan} and the distributions for all $d=2$ are consistent with those shown in \cite{conference}. Moreover, in the cases where \( d = 3, N = 5 \) and \( d = 4, N = 7 \), the maximum coherence is not 1, but rather \( \frac{N-d}{d} \). This observation led us to the following proposition. While this fact may be familiar to experts, we have not encountered it in the literature.  

\begin{proposition}
   For any $F \in \fun$, the coherence of $F$ is bounded above by $\min \{1, (N-d)/d\}$.
\end{proposition}
\begin{proof}
    Let $F \in \fun/(U(d) \times G)$. If $N \geq 2d$, then $\min\{(N-d)/d, 1\} = 1$. Since each frame vector of $F$ has squared norm equal to $1$, this provides the trivial bound on $\text{Coh}(F)$. Moreover, this bound is sharp: when $N \geq 2d$, it is possible to fill in the eigenstep array from Figure~\ref{fig:eigensteps} starting with $\mu_{1,1} = 1$, $\mu_{2,1} = 2$, so there is a frame having these eigensteps; for example, when $N=2d$ we can simply concatenate two copies of the same orthonormal basis. This choice of eigensteps corresponds to the situation when the first two frame vectors lie in the same line, and hence $\text{Coh}(F) = 1$.
    
    It remains to consider the case $N < 2d$, when $\min\{(N-d)/d,1\} = (N-d)/d$. Without loss of generality, assume that $\max_{i \neq j} |\<f_i,f_j\>|$ is attained by the first two frame vectors $f_1$ and $f_2$ i.e., $\text{Coh}(F) = |\<f_1, f_2\>|$. Recall that the eigensteps $\mu_{2,1}$ and $\mu_{2,2}$ are defined to be the eigenvalues of $S_2 = f_1f_1^* + f_2f_2^*$. These are also the eigenvalues of the upper left $2 \times 2$ minor of the Gram matrix $F^*F$, given by
    \begin{linenomath*}
    \begin{align*}
        [F^*F]_M := \begin{bmatrix}
            1 & \<f_1,f_2\> \\
            \<f_2, f_1\> & 1
        \end{bmatrix}.
    \end{align*}
    \end{linenomath*}
     The characteristic polynomial of $[F^*F]_M$ is
     \begin{linenomath*}
    \begin{align*}
        \chi_{[F^*F]_M}(\lambda) = \lambda^2 - 2\lambda + 1 - |\<f_1,f_2\>|^2.
    \end{align*}
    \end{linenomath*}
    Solving for its roots, the larger eigenvalue is
    \begin{linenomath*}
    \begin{align*}
        \mu_{2,1} = 1 + |\<f_1,f_2\>|.
    \end{align*}
    \end{linenomath*}
    From our work done in Figure~\ref{fig:eigensteps}, we know that $\mu_{2,1} \leq \mu_{N,1} = N/d$. Thus, we obtain
    \begin{linenomath*}
    \begin{align*}
        1 + |\<f_1,f_2\>| \leq \frac{N}{d}.
    \end{align*}
    \end{linenomath*}
    After rearranging, we arrive at the desired bound
    \begin{align*}
        \text{Coh}(F) = |\<f_1,f_2\>| \leq \frac{N-d}{d}.
    \end{align*}
\end{proof}

For \( d = 3 \) and \( N = 5 \), the fiber in \( \mathcal{F}_{5/3}^{3,5}(\boldsymbol{1})/(U(d) \times G) \) over a regular value of the momentum map is diffeomorphic to the two-torus \( U(1)^2 \). Thus, we are able to  visualize the coherence values on a fiber. To do so, we sample a regular value \( \mu \in \mathcal{P}_{3,5} \), and proceed as follows: first, we generate a mesh grid on \( [0,2\pi)^2 \); next, we map the grid onto frame space using \( \phi(\,\cdot\,, \Phi^{-1}(p)) : U(1)^2 \to\mathcal{F}_{5/3}^{3,5}(\boldsymbol{1})/(U(d) \times G) \); then, we compute the coherence at each grid point; and finally, we present the results as a heat map. The heat maps corresponding to two regular values are shown in \cref{fig: coh on fiber}.
\begin{figure}[h]
   \centering
    \includegraphics[width=0.45\linewidth]{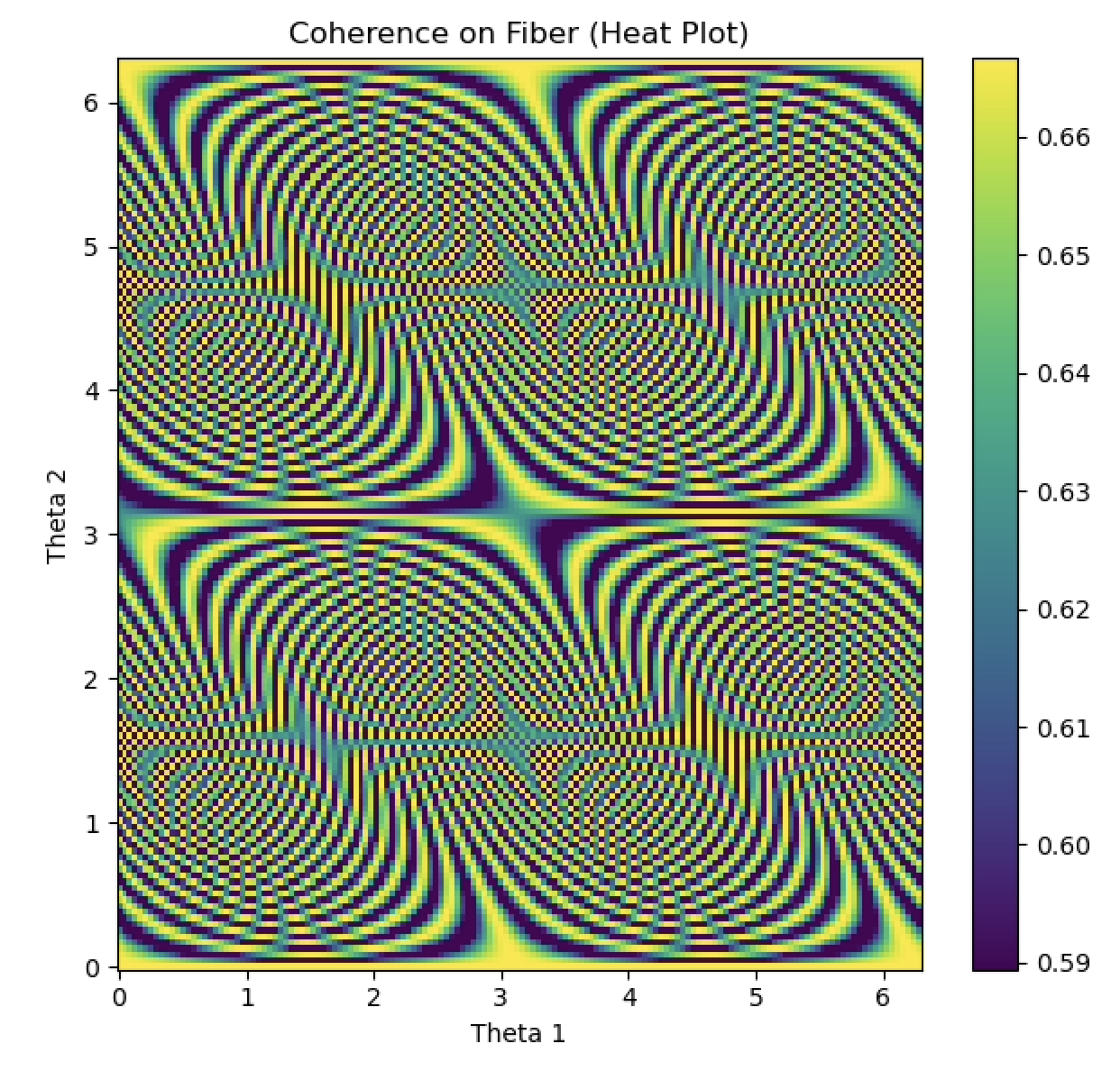}
    \includegraphics[width=0.45\linewidth]{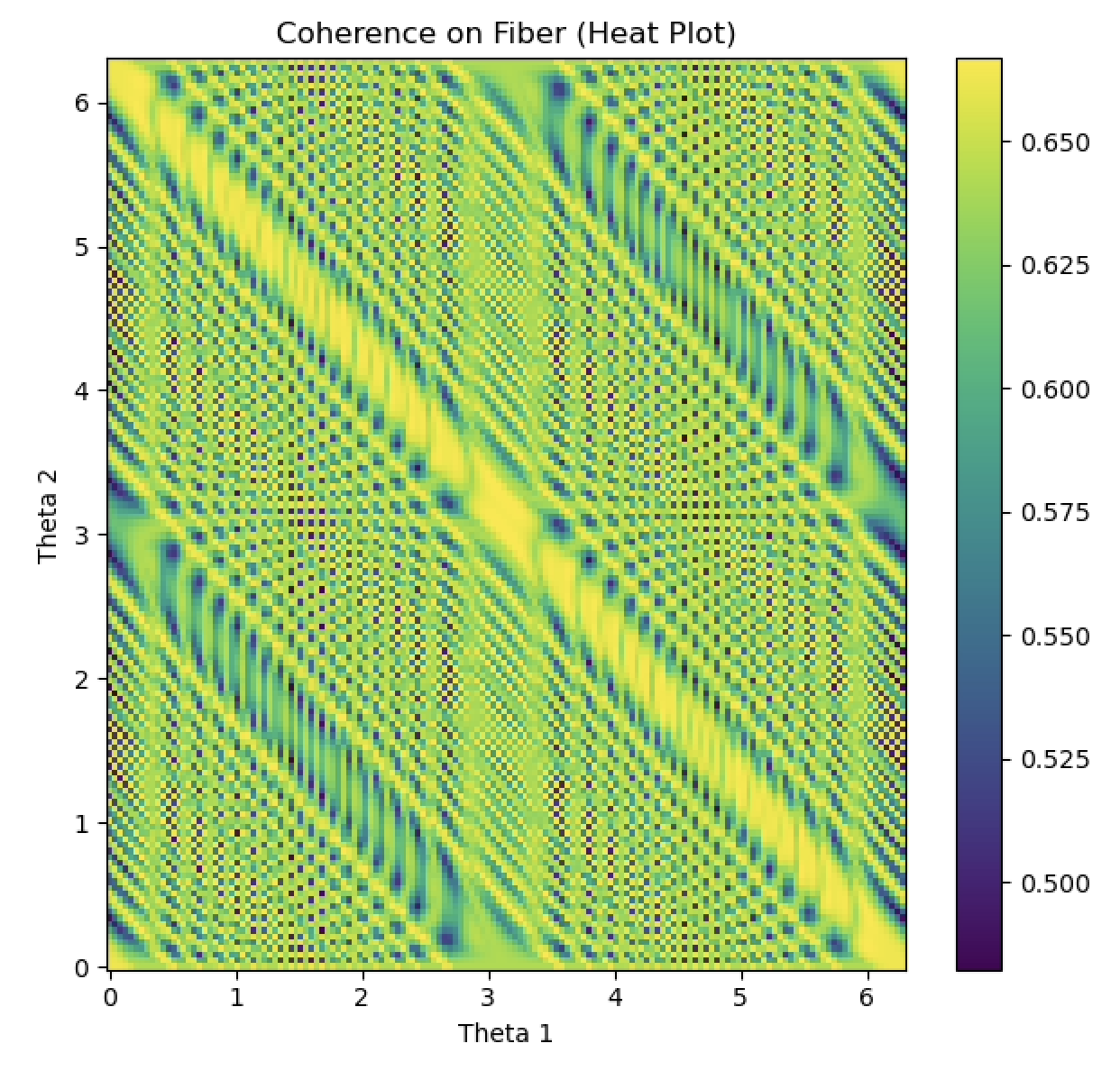}
    \caption{Coherence values on the fibers of the momentum map $\Phi : \mathcal{F}_{5/3}^{3,5}(\boldsymbol{1})/(U(d) \times G) \to \mathcal{P}_{3,5}$ over the two regular values $(1.589, 1.009)$ and $(1.361, 0.711)$ corresponding to the plots on the left and right, respectively. Note that the full eigenstep table for frames on these fibers can be found by plugging the regular values in for $x_1$ and $x_2$ in the table in Example~\ref{35example}.}
    \label{fig: coh on fiber}
\end{figure}

\section{Discussion}
\label{sec:conclusion}

In order to efficiently sample higher-dimensional FUNTF spaces, we need a better approach to sampling the eigenstep polytope than either rejection sampling (which seems to scale exponentially) or hit-and-run (which seems to converge too slowly to be practical). When $d=2$ there is a volume-preserving affine transformation $\overline{\mathcal{P}}_{2,N} \to [-1,1]^{d_T}$~\cite{cantarellaFastDirectSampling2016}; since $\frac{\vol \overline{\mathcal{P}}_{2,N}}{\vol [-1,1]^{d_T}} = O(N^{-3/2})$, rejection sampling becomes feasible even for very large $N$, and this gives a sampling algorithm for length-$N$ FUNTFs in $\C^2$ with expected runtime $\Theta(N^{5/2})$~\cite{conference}.\footnote{In fact, by more recent work~\cite{cantarellaFasterDirectSampling2024}, this can be speeded up to $\Theta(N^2)$.} Does such an affine tranformation exist for larger $d$? And what is the (asymptotic) volume of $\overline{\mathcal{P}}_{d,N}$?

Given some such way of speeding up the eigenstep sampling part of \eigenlift, it would then be interesting to perform numerical experiments on the distributions of various statistics of interest on FUNTFs. For example, what can be said about the distribution of coherence? The family of distributions observed in \cref{fig:coherences} seems hard to characterize without more data.

Haikin, Zamir, and Gavish~\cite{haikinRandomSubsetsStructured2017} have observed that for many classes of FUNTFs, the spectra of random submatrices follow a MANOVA distribution (see also~\cite{farrellLimitingEmpiricalSingular2011,edelmanBetaJacobiMatrixModel2008}). This, in turn, implies that these frames have better properties for applications like compressed sensing than random Gaussian frames. It seems quite reasonable to guess that the same is true for FUNTFs as a whole,\footnote{We're not sure to whom we ought to attribute this idea, but we first heard it from Dustin Mixon.} and this would be an interesting conjecture to test experimentally.

Finally, there is nothing particularly special about FUNTFs in our development of \eigenlift. Analogous algorithms should exist for complex frames with \emph{any} prescribed spectrum of the frame operator and \emph{any} prescribed frame vector norms. The quotient of such frame spaces by $U(d) \times G$ always contains an dense open subset which is toric, so \cref{sampling uniform} still applies.

\subsection*{Acknowledgments}

Thanks to Jason Cantarella, Tom Needham, Chris Manon, Emily King, and Dustin Mixon, for ongoing conversations about frames, symplectic geometry, and related topics. This work was supported by the National Science Foundation (DMS–2107700).

\bibliography{ref1}

\end{document}